\documentclass[11pt,reqno]{amsart}
\usepackage{amssymb, latexsym}
\usepackage[curve,matrix,arrow]{xy}
\usepackage{pifont}
\usepackage{multirow}
\newtheorem{defin}{Definition}[section]

\newtheorem{lemma}[defin]{Lemma}
\newtheorem{theorem}[defin]{Theorem}

\newtheorem{remark}[defin]{Remark}
\newtheorem{example}[defin]{Example}
\newtheorem{problem}[defin]{Problem}

\newcommand{\ra}{\rightarrow}

\newcommand{\ms}{\mapsto}
\newcommand{\ol}{\overline}
\newcommand{\im}{\mathrm{im}}

\newcommand{\N}{\mathbb N}

\newcommand{\FF}{\mathcal F}
\newcommand{\LL}{\mathcal L}
\newcommand{\TT}{\mathcal T}
\renewcommand{\SS}{\mathcal S}

\renewcommand{\L}{\mathcal L}
\newcommand{\Z}{\mathbb Z}
\newcommand{\Q}{\mathbb Q}
\newcommand{\R}{\mathcal R}
\renewcommand{\O}{\mathcal O}

\newcommand{\GL}{{\mathrm{GL}}}

\newcommand{\Aut}{{\mathrm{Aut}}}
\newcommand{\End}{{\mathrm{End}}}
\newcommand{\Hom}{{\mathrm{Hom}}}

\newcommand{\Stab}{{\mathrm{Stab}}}
\newcommand{\SL}{{\mathrm{SL}}}
\newcommand{\PGL}{{\mathrm{PGL}}}

\newcommand{\id}{{\mathrm{id}}}
\newcommand{\rank}{{\mathrm{rank}}}

\newcommand{\MySigma}{\Gamma}

\newenvironment{items}{\begin{list}{$\alph{item})$}
{\labelwidth18pt \leftmargin18pt \topsep3pt \itemsep1pt \parsep0pt}}
{\end{list}}

\begin{document}

\title{The conjugacy problem in $\GL(n,\Z)$}

\author{Bettina Eick} 
\address{
Institut Computational Mathematics,
Technische Universit\"at Braunschweig, 	
38106 Braunschweig, Germany}
\email{beick@tu-bs.de}

\author{Tommy Hofmann}
\address{
Fachbereich Mathematik, 
Technische Universtit\"at Kaiserslautern, 
67653 Kaisers\-lautern, Germany}
\email{thofmann@mathematik.uni-kl.de}

\author{E.A.\ O'Brien}
\address{
Department of Mathematics, 
University of Auckland,
Auckland, New Zealand}
\email{e.obrien@auckland.ac.nz}

\thanks{Eick and O'Brien were both supported by the Alexander von 
Humboldt Foundation, Bonn, and by the Marsden Fund of New Zealand
via grant UOA 1323. They thank the Hausdorff Institute, Bonn, 
for its hospitality while this work was completed.
Hofmann was supported by Project II.2 of SFB-TRR 195 `Symbolic Tools in
Mathematics and their Application'
of the German Research Foundation (DFG). 
We thank Stefano Marseglia 
and Gabriele Nebe for helpful comments.} 

\subjclass[2010]{Primary 20C15; Secondary 20D05, 11Y40}

\date{\today}

\begin{abstract}
We present a new algorithm that, given two matrices in $\GL(n,\Q)$, 
decides if they are conjugate in $\GL(n,\Z)$ and, if so, determines a 
conjugating matrix. We also give an algorithm to construct a generating 
set for the centraliser in $\GL(n,\Z)$ of a matrix in $\GL(n,\Q)$. We do 
this by reducing these problems respectively to the isomorphism
and automorphism group problems for certain modules over rings of the
form $\O_K[y]/(y^l)$, where $\O_K$ is the maximal order of an algebraic 
number field and $l \in \N$, and then provide algorithms to solve the 
latter. The algorithms are practical and our implementations are publicly 
available in {\sc Magma}.  
\end{abstract}
\maketitle

\section{Introduction}

Let $T$ and $\hat{T}$ be elements of $\GL(n,\Q)$. The {\it rational 
conjugacy problem} asks if there exists $X \in \GL(n,\Q)$ such that 
$XTX^{-1} = \hat{T}$. It is well-known that this can be decided 
effectively by computing and comparing the rational canonical forms 
of $T$ and $\hat{T}$. 
More difficult is the {\em integral conjugacy problem}: decide whether
or not there  exists $X \in \GL(n,\Z)$ with $X T X^{-1} = \hat{T}$. 
Clearly, if $T$ and $\hat{T}$ are not conjugate in $\GL(n,\Q)$, then 
they are not conjugate in $\GL(n,\Z)$, but the converse does not hold.
Associated to the integral conjugacy problem is the {\em centraliser 
problem}: determine a generating set for $C_\Z(T) = \{ X \in \GL(n,\Z) 
\mid X T X^{-1} = T\}$. Since $C_\Z(T)$ is arithmetic, 
the work of Grunewald and Segal \cite{GS80} implies that it is both 
finitely generated and finitely presented.  But no practical algorithm 
to compute a finite generating set for an arithmetic group is known. 

Grunewald \cite{Gru80} proved that the integral conjugacy and centraliser 
problems are decidable. We recall the basic ideas of this proof briefly.
Let $\O_K$ denote the maximal order of the algebraic number field $K$,
let $l \in \N$ and denote $P_l(\O_K) = \O_K[y]/(y^l)$. A $P_l(\O_K)$-module
$M$ is integral if $M \cong \Z^n$ for some $n \in \N$ as abelian group.
Grunewald described how to reduce the integral conjugacy and centraliser
problems respectively to the isomorphism and automorphism group problems of integral 
$P_l(\O_K)$-modules. He proved that such a module has submodules of finite 
index which have a certain integral canonical form,
and exhibited how
the isomorphism and automorphism group problems can be solved using these
{\it standard} submodules.  A critical 
weakness of his method is that all standard submodules of an 
unknown index in a given module must be constructed, but no practical 
algorithm for this purpose was provided. 

Let $M$ be a $P_l(\O_K)$-module. 
Define $K_i(M) = \{ m \in M \mid m y^i = 0\}$ for $0 \leq i \leq l + 1$
and $L_i(M) = K_{i+1}(M)y +K_{i-1}(M)$ for $1 \leq i \leq l$.
Now $M$ is {\em standard} if it is integral and 
$Q_i(M) = K_i(M)/L_i(M)$ is free as $\O_K$-module for 
$1 \leq i \leq l$. We investigate the structure of standard 
$P_l(\O_K)$-modules in detail and exhibit effective methods to 
solve their isomorphism and automorphism group problems; this also corrects 
Grunewald's solution to the latter problem. We then 
investigate the standard submodules of an arbitrary 
integral $P_l(\O_K)$-module.  Our main result is the following.
\begin{theorem}
\label{main}
Let $M$ be an integral $P_l(\O_K)$-module and let $S$ be a standard 
submodule of finite index in $M$. Let $r_i$ be the rank of $Q_i(S)$ 
as free $\O_K$-module for $1 \leq i \leq l$.
\begin{itemize}
\item[\rm (a)]
The map $$K_i(S)/L_i(S) \ra K_i(M)/L_i(M) : a + L_i(S) \ms a + L_i(M)$$
induces an embedding
of $Q_i(S)$ into $Q_i(M)$ with image $\ol{Q}_i(S)$, say.
\item[\rm (b)]
If $[Q_i(M):\ol{Q}_i(S)] = h_i$ for $1 \leq i \leq l$, then $[M:S] = 
h_1 h_2^2 \cdots h_l^l$.
\item[\rm (c)]
There are exactly 
$$ \prod_{i = 1}^l \frac{(h_i \cdots h_l)^{(r_i + \ldots + r_l)}}{ h_i^{r_i}}$$ 
standard
submodules $T$ in $M$ with $\ol{Q}_i(T) = \ol{Q}_i(S)$ for $1 \leq i \leq l$.
\end{itemize}
\end{theorem}

Our proof, 
given in Theorems \ref{goodexist} and 
\ref{goodcount},
is constructive and translates to practical algorithms to construct
one or all standard submodules of minimal index in an integral
$P_l(\O_K)$-module.
We use our methods to obtain practical algorithms to solve the 
isomorphism and automorphism group problems for these modules, 
and, in turn, also
to solve the integral conjugacy and centraliser problems.

In Section \ref{transl} we briefly recall the translation of the integral
conjugacy and centraliser problems to the isomorphism and automorphism
group problems for integral $P_l(\O_K)$-modules. 
In Section \ref{max-orders} we present a new algorithm to construct (one or 
all) standard submodules of smallest index in an integral $P_l(\O_K)$-module,
and so solve the isomorphism and automorphism 
group problems for these modules. 
In Sections 4 and 5 we give pseudo-code outlines of our algorithms. In 
Section 6 we comment on our implementations of these algorithms, available 
in {\sc Magma} \cite{MAGMA}, and illustrate their performance
and limitations. Finally we discuss variations of our 
methods and open problems.
Where appropriate, we cite corresponding statements from \cite{Gru80}. 
In most cases, we also include independent proofs 
to ensure that our work is reasonably self-contained and the language 
employed is consistent. 

We conclude the introduction by mentioning related work.
Latimer and MacDuffee \cite{LMa1933} 
solved the integral conjugacy problem for matrices with 
irreducible characteristic polynomial; 
Husert \cite{Hus16} solved it 
for nilpotent and semisimple matrices;
Marseglia~\cite{Mars2018} solved it 
for matrices with squarefree characteristic polynomial. 
Opgenorth, Plesken and Schulz 
\cite{OPSc1998} developed algorithms to solve both 
problems for matrices of finite order. 
Sarkisjan \cite{Sar79} exhibited a method to solve the simultaneous integral 
conjugacy problem for lists of matrices. 
Karpenkov \cite{Kar13} described the set of conjugacy classes in $\SL(n, \Z)$.

\section{Translation to modules over truncated polynomial rings}
\label{transl} 

We now describe how to reduce the integral conjugacy problem and the 
centraliser problem to the isomorphism and automorphism group problems 
for a certain type of module defined over the maximal order of an
algebraic number field. 

\subsection{Reduction to integral matrices}

Recall that $S \in M_n(\Q)$ is semisimple if the natural 
$\Q S$-module $V = \Q^n$ is a direct sum of irreducible $\Q S$-modules,
where $\Q S$ denotes the $\Q$-subalgebra of $M_n(\Q)$ generated by $S$. 
A matrix $U$ in $M_n(\Q)$ is nilpotent if there exists $l \in \N$
with $U^l = 0$. By the Jordan-Chevalley decomposition, 
for every $T \in \GL(n,\Q)$ there exist unique $S \in 
\GL(n,\Q)$ and $U \in M_n(\Q)$ such that $S$ is semisimple, 
$U$ is nilpotent, $T = S + U$, and $SU=US$.  

Although the input matrices 
to the integral conjugacy 
problem are rational, we can readily reduce to integral matrices.

\begin{lemma}\cite[Lemma 1]{Gru80}
\label{integral}
Let $T, \hat{T} \in \GL(n,\Q)$ with Jordan-Chevalley decompositions 
$T = S+U$ and $\hat{T} = \hat{S} + \hat{U}$. 
Choose $k \in \N$ so that $kS, kU, k\hat{S},$ and $k\hat{U}$ are integral
matrices.
\begin{items}
\item[\rm (a)]
$kT = kS + kU$ 
and $k\hat{T} = k\hat{S} + k\hat{U}$ are Jordan-Chevalley decompositions. 
\item[\rm (b)]
$T$ and $\hat{T}$ are conjugate in $\GL(n,\Z)$ if and only
if $kT$ and $k \hat{T}$ are conjugate in $\GL(n,\Z)$. 
\item[\rm (c)]
$C_\Z(T) = C_\Z(kT)$. 
\end{items}
\end{lemma}

Our first step in solving the integral conjugacy and centraliser problems 
for matrices $T, \hat{T} 
\in \GL(n,\Q)$ is to choose $k \in \N$ as in Lemma~\ref{integral} and to 
replace $T, \hat{T}$ by $kT, k\hat{T}$ respectively. Now $kT$ and  
$k\hat{T}$ have {\em integral Jordan-Chevalley decompositions}. 

\subsection{Translation to modules} 

Let $T, \hat{T} \in \GL(n,\Q)$ have integral Jordan-Chevalley decompositions
$T = S+U$ and $\hat{T} = \hat{S} + \hat{U}$. Choose $l \in \N$ minimal
with $U^l = 0$ and let $P(x)$ be the minimal polynomial of $S$. 
We may assume that $P(x)$ is also the minimal polynomial
of $\hat{S}$ and that $\hat{U}^l = 0$: otherwise $T$ and $\hat{T}$ 
are not conjugate in $\GL(n,\Q)$ and so not in $\GL(n,\Z)$.

Let $R = \Z[x]/(P(x))$ and $P_l(R) = R[y]/(y^l) = \Z[x,y]/(P(x),y^l)$, so
that $P_l(R)$ is a truncated polynomial ring over the commutative ring $R$.
Hence $P_l(R) = \{ r_0 y^0 + \dots + r_{l-1} y^{l-1} + (y^l) \mid 
r_i \in R \}$. We embed $R$ into $P_l(R)$ via $r \ms r y^0 + (y^l)$.

A $P_l(R)$-module $M$ is {\em integral} if $M \cong \Z^n$ 
for some $n \in \N$ as abelian group. In this case we fix a $\Z$-basis
in $M$ and identify $M$ with $\Z^n$ and hence $\Aut(M)$ with $\GL(n,\Z)$
and $\End(M)$ with $M_n(\Z)$. 

\begin{lemma}\cite[Lemma 3]{Gru80} 
\label{translate}
Let $M = \Z^n = \hat{M}$ as abelian groups.
\begin{enumerate}
\item[(a)]
$T$ induces the structure of a $P_l(R)$-module on $M$ via $vx = vS$
and $vy = vU$ for $v \in M$. Similarly, $\hat{M}$ is a $P_l(R)$-module 
for $\hat{T}$.
\item[(b)]
$T$ and $\hat{T}$ are conjugate in $\GL(n,\Z)$ if and only if
$M$ and $\hat{M}$ are isomorphic as $P_l(R)$-modules.
\item[(c)]
$C_\Z(T) = \Aut_{P_l(R)}(M) = \{ X \in \GL(n,\Z) \mid X\ol{a} = \ol{a}X 
\mbox{ for } a \in P_l(R) \}$, where $P_l(R) \ra M_n(\Z) : a \ms 
\ol{a}$ is the action of $P_l(R)$ on $M = \Z^n$.
\end{enumerate}
\end{lemma}

It remains to solve the isomorphism and automorphism group problems
for integral $P_l(R)$-modules. 

Let $P(x) = P_1(x) \cdots P_r(x)$ be the factorisation of $P(x)$ into
irreducible polynomials over $\Q[x]$. Since $S$ is semisimple and integral, 
$P_1(x), \dots, P_r(x)$ are pairwise distinct, monic, and integral.
For $1 \leq i \leq r$ let $R_i = \Z[x]/(P_i(x))$, let $p_i(x) = P(x)/P_i(x)
\in \Z[x]$, let $M_i$ be the image of $p_i(S)$ in $M$ and $\hat{M}_i$ the
image of $p_i(\hat{S})$ in $\hat{M}$. 
Now $M_i$ and $\hat{M}_i$ are both 
$P_l(R)$- and $P_l(R_i)$-modules.
Write $D = M_1 + \dots + M_r \leq M$ and 
$\hat{D} = \hat{M}_1 + \dots + \hat{M}_r \leq \hat{M}$.  By construction,
$D$ is a subgroup of finite index, say $d$, in $M = \Z^n$ and, similarly, 
$\hat{D}$ is a subgroup of finite index, say $\hat{d}$, in $\hat{M} = \Z^n$. 
Further, $D$ and $\hat{D}$ are $P_l(R)$-submodules of $M$ and $\hat{M}$ 
respectively. Choose $c \in \N$ such that $cM \leq D$; 
one option is $c=d$. 

\begin{theorem} \cite[Lemma 8]{Gru80}
\label{module}
Let $T, \hat{T} \in \GL(n,\Q)$ have integral Jordan-Chevalley decompositions.
\begin{enumerate}
\item[(a)] 
The following are equivalent.
\begin{enumerate}
\item[(i)]
$T$ and $\hat{T}$ are conjugate in $\GL(n,\Z)$.
\item[(ii)]
$M$ and $\hat{M}$ are isomorphic as $P_l(R)$-modules and $d = \hat{d}$.
\item[(iii)]
There exists a $P_l(R)$-module isomorphism $\gamma : D \ra \hat{D}$
with $\gamma(cM) = c \hat{M}$.
\end{enumerate}
\item[(b)]
$C_\Z(T) = \Aut_{P_l(R)}(M)$ is isomorphic to the stabiliser of 
$cM$ in $\Aut_{P_l(R)}(D)$.
\end{enumerate}
\end{theorem}

\begin{proof}
\leavevmode
\begin{enumerate}
\item[(a)]
Translating notation shows that (i) is equivalent to (ii). 
We now prove that (ii) is equivalent to (iii). 
If $\gamma : M \ra \hat{M}$ is an isomorphism of  
$P_l(R)$-modules, then its restriction $\gamma : D \ra \hat{D}$ is
also an isomorphism of $P_l(R)$-modules, as $D$ and $\hat{D}$ are both
fully invariant submodules. Similarly, $\gamma(cM) = c \gamma(M) = c \hat{M}$. 
It remains to show the converse. Let $\gamma : D \ra \hat{D}$ be an
isomorphism of $P_l(R)$-modules with $\gamma(cM) = c \hat{M}$. 
Now 
\[ \delta : M \ra \hat{M} : w \ms c^{-1} \gamma(cw)\]
is also an isomorphism of $P_l(R)$-modules.
\item[(b)]
Translating notation establishes that $C_\Z(T) = \Aut_{P_l(R)}(M)$. 
We now establish the isomorphism. Let 
\[ \varphi : \Aut_{P_l(R)}(M) \ra \Aut_{P_l(R)}(D)\]
be induced by restriction. This is well-defined, since $D$ 
is a fully invariant submodule by construction. Note that $\varphi$ is
injective, since $D$ has finite index in $M$ and contains $cM$.
Each automorphism of $M$ leaves $c M$ invariant; 
thus  ${\rm im}(\varphi)$ is contained
in the stabiliser $S$ of $cM$ in $\Aut_{P_l(R)}(D)$.
It remains to show that ${\rm im}(\varphi) = S$.
Let $\gamma \in S$ and let $w \in M$. Define
\[ \delta : M \ra M : w \ms c^{-1} \gamma(cw).\]
Observe $\delta \in \Aut_{P_l(R)}(M)$ and $\varphi(\delta) = \gamma$.
Thus $\varphi$ is surjective. \qedhere
\end{enumerate}
\end{proof}

Theorem \ref{module} reduces the integral conjugacy and centraliser 
problems for $T$ and $\hat{T}$ to the construction of generators for 
$A = \Aut_{P_l(R)}(D)$ and the determination of an arbitrary $P_l(R)$-module 
isomorphism $\phi : D \ra \hat{D}$.  If $A$ and $\phi$ are given, then
the orbit $O$ of $cM$ under $A$ can be constructed:  $O$ is finite since 
$cM$ has finite index in $M$. 
An isomorphism $\gamma$ as in Theorem \ref{module}(a)(iii) exists 
if and only if $c \hat{M}$ is contained in 
$\{ \phi(o) \mid o \in O\}$. Generators for $\Stab_A(cM)$ as
required in Theorem~\ref{module}(b) can be constructed from 
$A$ via Schreier generators, see \cite[Section 4.1]{HEO05}.

\begin{theorem}\cite[Lemma 7]{Gru80}
\label{dirsplit}
\mbox{}
\begin{enumerate}
\item[\rm (a)]
$D \cong \hat{D}$ as $P_l(R)$-modules if and only if $M_i \cong \hat{M}_i$
as $P_l(R_i)$-modules for $1 \leq i \leq r$.
\item[\rm (b)]
$\Aut_{P_l(R)}(D) = \Aut_{P_l(R_1)}(M_1) \times \dots \times 
\Aut_{P_l(R_r)}(M_r)$.
\end{enumerate}
\end{theorem}

\begin{proof}
Both statements follow readily from the fact that $D$ and $\hat{D}$
are direct sums of $M_1, \ldots, M_r$ and $\hat{M}_1, \ldots, \hat{M}_r$,
respectively.  It remains to prove this direct sum property. 
Observe that $M_i = \im(p_i(S)) \leq \ker(P_i(S))$, 
since $p_i(x) P_i(x) = P(x)$ and
$P(S) = 0$. 
Consider $m \in M_i \cap M_j$ where $i \not=j$. Now 
$m \in \im(p_i(S)) \cap \ker(P_j(S))$. As $i \neq j$, it follows that 
$P_j(x) \mid p_i (x)$ and hence $m = 0$.
\end{proof}

Thus a solution to the isomorphism and automorphism group problems for 
the $P_l(R)$ modules $M$ and $\hat{M}$ reduces to the same 
problems for the $P_l(R_i)$-modules $M_i$ and $\hat{M}_i$
 for $1 \leq i \leq r$.
Since $P_i(x)$ is irreducible, 
$R_i$ is an order in the algebraic number field $\Q[x]/(P_i(x))$.

\subsection{Reduction to modules over maximal orders}
\label{maxord}

Let $K$ be an algebraic number field with maximal order $\O_K$, let
$\O$ be an arbitrary order in $\O_K$, and  
let $l \in \N$.  
Let $M$ be an integral $P_l(\O)$-module, so $M = \Z^n$ for some
$n \in \N$. We embed $M$ in the vector space $V = \Q^n$. 
Thus $\O$ acts on $V$ and induces an action by its quotient field
$K$ on $V$. Hence the maximal order $\O_K$ acts on $V$. 
Note that $M$ need not  
be closed under the action of $\O_K$, but it has subgroups 
which are, and we call these $P_l(\O_K)$-submodules of $M$. 

\begin{lemma}
\label{unique}
$M$ has a unique maximal $P_l(\O_K)$-submodule of finite index.
\end{lemma}

\begin{proof}
Note that $\O$ has finite index in $\O_K$ as additive abelian group. Let 
$T$ be a transversal for $\O$ in $\O_K$. For each $t \in T$ the sublattice 
$M t$ is an integral $P_l(\O)$-module in $V$ and $L := \cap \{ Mt \mid t 
\in T\}$ is a $P_l(\O_K)$-submodule of finite index in $M$. 

Let $N$ be an arbitrary $P_l(\O_K)$-submodule of finite index in 
$M$. Hence $N = Nt$ for each $t \in T$,  so $N \leq L$. Thus $L$ is the
unique maximal $P_l(\O_K)$-submodule of finite index in $M$.
\end{proof}

\begin{remark}\label{findoksub}
{\rm A basis for 
this unique maximal $P_l(\O_K)$-submodule $L$ of $M$ can be computed readily.
Let $B$ be the standard $\Z$-basis of $M$ and let $W = \{w_1, \dots, w_e \}$ 
be a $\Z$-basis of $\O_K$. With respect to $B$, each $w \in W$ acts 
via a rational matrix $C_w \in M_n(\Q)$ on $M$. Let $C = 
(C_{w_1} | C_{w_2} | \dots | C_{w_e})$ be the rational $n \times en$ matrix
obtained by concatenating the matrices $C_{w_1}, \dots, C_{w_e}$. Let $d$ 
be the smallest positive integer so that $dC$ is an integral matrix. Let 
$P \in \GL(n,\Z)$ and $Q \in \GL(en, \Z)$ satisfy $P (dC) Q = D$ where $D$ 
is a diagonal matrix with diagonal $(d_1, \dots, d_n)$ satisfying $d_i \mid 
d_{i+1}$ for $1 \leq i \leq n-1$. Let $A \in M_n(\Z)$ be the diagonal matrix 
with diagonal $(a_1, \dots, a_n)$, where $a_i$ is the denominator of 
$d_i/d$. Now $A P$ is a basis for $L$. The necessary computations are part 
of the standard Smith normal form algorithm, see \cite[Section 9.3]{HEO05}.}
\end{remark}

\begin{theorem}
\label{reduct}
Let $M$ and $\hat{M}$ be $P_l(\O)$-modules. Let $L$ and $\hat{L}$
denote their unique maximal $P_l(\O_K)$-submodules of finite index. 
Let $c \in \N$ be such that $cM \leq L$. 
\begin{items}
\item[\rm (a)]
$M \cong \hat{M}$ as $P_l(\O)$-modules if and only if there exists a
$P_l(\O_K)$-module isomorphism $\gamma : L \ra \hat{L}$ with 
$\gamma(cM) = c \hat{M}$.
\item[\rm (b)]
$\Aut_{P_l(\O)}(M)$ is isomorphic to the stabiliser of $cM$ in 
$\Aut_{P_l(\O_K)}(L)$.
\end{items}
\end{theorem}

\begin{proof}
\leavevmode
\begin{enumerate}
\item[(a)]
Let $\rho : M \ra \hat{M}$ be a $P_l(\O)$-module isomorphism,
so $\rho(cM) = c \hat{M}$. If $a \in \O_K$, then there 
exists $b \in \Z$ such that $ba \in \O$. If $v \in L$, then 
$\rho(v a) =  b^{-1} b \rho(v a) = b^{-1} \rho( v ba) = b^{-1} ba 
\rho(v) = a \rho(v)$. Hence the restriction of $\rho$ to $L$ is an 
$\O_K$-module homomorphism that maps $L$ onto a $P_l(\O_K)$-submodule in 
$\hat{M}$. But $\hat{L}$ is the unique maximal $P_l(\O_K)$-submodule 
in $\hat{M}$,
so $\rho : L \ra \hat{L}$ is a $P_l(\O_K)$-module isomorphism.
Conversely, let $\gamma : L \ra \hat{L}$ be a $P_l(\O_K)$-module isomorphism
with $\gamma(c M) = c \hat{M}$. Define $\sigma : M \ra \hat{M}$ via
$\sigma(w) = c^{-1} \gamma(cw)$. Now $\sigma$ is a $P_l(\O)$-module
isomorphism from $M$ to $\hat{M}$.
\item[(b)] 
This follows by arguments similar to the proof of Theorem \ref{module} (b). 
\qedhere
\end{enumerate}
\end{proof}

Theorem~\ref{reduct} reduces the isomorphism and automorphism group
problems for $P_l(\O)$-modules to the same problems for $P_l(\O_K)$-modules
via orbit-stabiliser constructions similar to those for Theorem~\ref{module}.

\section{Modules defined over maximal orders}
\label{max-orders} 

Let $K$ be an algebraic number field with maximal order $\O_K$ and let
$l \in \mathbb{N}$. Write $P_l(\O_K) = \O_K[y]/(y^l)$ for the truncated 
polynomial ring over $\O_K$. Our goal is to develop the necessary theory 
to underpin {\em practical} and {\em constructive} algorithms to solve the
isomorphism and automorphism group problems for integral $P_l(\O_K)$-modules. 

\subsection{Some structure theory}
\label{struct}

Let $M$ be an integral $P_l(\O_K)$-module. 
Let $\rank(A)$ denote the torsion free rank of an abelian group $A$. 
Define
\begin{eqnarray*}
K_i(M) &=& \{ m \in M \mid m y^i = 0 \} 
      \;\; \mbox{ for } 0 \leq i \leq l + 1, \\
I_{i,j}(M) &=& K_j(M) y^{j-i} + K_{i-1}(M) 
      \;\; \mbox{ for } 1 \leq i \leq l \mbox{ and } 
           i \leq j \leq l+1, \\
L_i(M) &=& I_{i, i+1}(M) = K_{i+1}(M) y + K_{i-1}(M) 
      \;\; \mbox{ for } 1 \leq i \leq l, \\
Q_i(M) &=& K_i(M) / L_i(M) 
      \;\; \mbox{ for } 1 \leq i \leq l, \mbox{ and } \\
r_i &=& \rank(Q_i(M)) / \rank(\O_K) \mbox{ for } 1 \leq i \leq l.
\end{eqnarray*}
The sequence of integers $(r_1, \dots, r_l)$ is the {\em type} of $M$. 
The series of $P_l(\O_K)$-submodules
\[ M = K_{l+1}(M) = K_l(M) \geq K_{l-1}(M) \geq \dots \geq K_0(M) = \{0\}\]
is refined for $1 \leq i \leq l$ by the $P_l(\O_K)$-submodule series
\[ K_i(M) = I_{i,i}(M) 
   \geq I_{i,i+1}(M) \geq \dots \geq I_{i,l+1}(M) = K_{i-1}(M).\]
Let $(\Sigma)$ denote the resulting refined series of 
$P_l(\O_K)$-submodules $I_{i,j}(M)$ through $M$. Observe
that $(\Sigma)$ is fully invariant under $P_l(\O_K)$-isomorphisms by 
construction. The following lemma describes the quotients of $(\Sigma)$.

\begin{lemma}
\label{modseries}
Let $M$ be an integral $P_l(\O_K)$-module. Then $y$ acts trivially on 
each quotient $K_i(M)/K_{i-1}(M)$ and thus on each quotient of $(\Sigma)$.
Further $I_{i,j}(M) / I_{i,j+1}(M) \cong Q_j(M)$ as $\O_K$-modules for 
$1 \leq i \leq j \leq l$.
\end{lemma}

\begin{proof}
Let $i \in \{1, \dots, l\}$ and $a \in K_i(M)$. Observe that
 $a y y^{i-1} = a y^i 
= 0$ so $ay \in K_{i-1}(M)$. Hence $y$ acts trivially on $K_i(M)/
K_{i-1}(M)$ and on the quotients of $(\Sigma)$. Define 
$$\sigma_{i,j} : K_j(M) \ra I_{i,j}(M) / I_{i,j+1}(M) : a \ms a y^{j-i} 
+ I_{i,j+1}(M).$$ 
Thus $\sigma_{i,j}$ is surjective by construction. We 
show that $\ker(\sigma_{i,j}) = L_j(M)$.
If $a \in L_j(M)$, then $a = by + c$ for some $b \in K_{j+1}(M)$
and $c \in K_{j-1}(M)$. Thus $a y^{j-i} = by^{j-i+1} + cy^{j-i}$. 
But 
$b y^{j-i+1} \in K_{j+1}(M) y^{j-i+1}$ and $c y^{j-i} \in K_{i-1}(M)$,
so $a y^{j-i} \in I_{i,j+1}(M)$. Thus $a \in \ker(\sigma_{i,j})$.
Conversely, if $a \in \ker(\sigma_{i,j})$, then $a y^{j-i} \in I_{i,j+1}(M)$.
Thus $a y^{j-i} = b y^{j+1-i} + c$ for some $b \in K_{j+1}(M)$ and 
$c \in K_{i-1}(M)$. Hence $(a-by) y^{j-i} = c \in K_{i-1}(M)$. Therefore
$$(a-by) y^{j-1} = (a-by) y^{j-i} y^{i-1} = c y^{i-1} = 0,$$ 
so $a-by = x \in K_{j-1}(M)$. Hence $a = by + x \in K_{j+1}(M) y + K_{j-1}(M)
= L_j(M)$.
\end{proof}

A $P_l(\O_K)$-module $M$ is {\em standard} if it is integral and its
quotients $Q_1(M), \dots, Q_l(M)$ are 
free as $\O_K$-modules, and so torsion free as abelian groups. 
Standard modules play a key role in 
our algorithm. We now investigate them in more detail.

\subsection{Standard modules and their isomorphisms and automorphisms}
\label{standardautiso}

Let $M$ be a standard $P_l(\O_K)$-module of type $(r_1, \dots, r_l)$. 
For $1 \leq j \leq l$ let $\FF_j$ be a set of preimages of a free 
generating set for $Q_j(M)$ as $\O_K$-module under the natural
epimorphism $\varphi_j : K_j(M) \ra Q_j(M)$. Let $U_j$ and $W_j$ be 
the $\O_K$-submodule and the $P_l(\O_K)$-module generated by $\FF_j$,
respectively. The following theorem asserts that 
the elements of the sequence $\FF = (\FF_1, \dots, \FF_l)$ 
generate $M$ as $P_l(\O_K)$-module.
We call $\FF$ a {\em standard generating sequence} for $M$. 

\begin{theorem}
\label{standnf}
Let $M$ be a standard $P_l(\O_K)$-module of type $(r_1, \dots, r_l)$.
\begin{enumerate}
\item[\rm (a)]
$K_j(M) = U_j \oplus L_j(M)$ and
$I_{i,j}(M) = U_j y^{j-i} \oplus I_{i,j+1}(M)$ for $1 \leq i \leq l$
and $i \leq j \leq l$.
\item[\rm (b)]
$U_j y^k \cong Q_j(M)$ as $\O_K$-modules for $1 \leq j \leq l$ and
$0 \leq k \leq j-1$.
\item[\rm (c)]
$W_j = U_j \oplus U_j y \oplus \dots \oplus U_j y^{j-1}$ as $\O_K$-module
for $1 \leq j \leq l$.
\item[\rm (d)]
$W_j \cong P_j(\O_K)^{r_j}$ as $P_l(\O_K)$-modules for $1 \leq j \leq l$.
\item[\rm (e)]
$M = W_1 \oplus \dots \oplus W_l$ as $P_l(\O_K)$-module.
\end{enumerate}
\end{theorem}

\begin{proof}
\mbox{}
\begin{enumerate}
\item[(a)]
Since $Q_j(M)$ is free, the natural epimorphism $K_j(M) \ra Q_j(M)$ with
kernel $L_j(M)$ splits. The construction of $U_j$ implies that $K_j(M)
= U_j \oplus L_j(M)$.
The proof of Lemma \ref{modseries} asserts that 
$$\sigma_{i,j} : K_j(M)
\ra I_{i,j}(M) / I_{i,j+1}(M) : a \ms a y^{j-i} + I_{i,j+1}(M)$$ is
surjective with kernel $L_j(M)$. Hence the restriction of $\sigma_{i,j}$
to $U_j$ is an isomorphism of the form $U_j \ra I_{i,j}(M)/I_{i,j+1}(M):
a \ms a y^{j-i} + I_{i,j+1}(M)$. Thus $I_{i,j}(M) = U_j y^{j-i}
+ I_{i,j+1}(M)$. Finally, the sum $U_j y^{j-i} + I_{i,j+1}(M)$ is direct,
since $I_{i,j}(M)/I_{i,j+1}(M) \cong Q_j(M)$ is free.

\item[(b)]
This follows from (a) and Lemma \ref{modseries}.

\item[(c)]
Fix $j \in \{1, \dots, l\}$. By definition, $W_j$ is the $P_l(\O_K)$-module
generated by the $\O_K$-module $U_j$. As $W_j \leq K_j(M)$, this yields that
$W_j = U_j + U_jy \dots + U_j y^{j-1}$. 
Further, $K_j(M) = U_j \oplus L_j(M)$ by (a). 
If $k > 0$ then 
$U_j y^k \leq K_{j-1}(M) \leq L_j(M)$, so  
$U_j y^k \cap U_j = \{0\}$. Hence $U_j y^{k+h} \cap
U_j y^k = \{0\}$ for $h \geq 0$, so $W_j$ is the direct sum of the
$\O_K$-submodules $U_j y^k$.

\item[(d)]
Fix $j \in \{1, \dots, l\}$. As $W_j \leq K_j(M)$, it follows that
$W_j y^j = \{0\}$ and $W_j$ is a $P_j(\O_K)$-module. Since $W_j$ is
generated by the $r_j$ elements in $\FF_j$, it is
an epimorphic image of the free module $P_j(\O_K)^{r_j}$. Now (a)
implies that the ranks of $W_j$ and $P_j(\O_K)^{r_j}$ as $\O_K$-modules
agree, thus the two modules are isomorphic.

\item[(e)]
Repeated application of (a) shows that 
\begin{eqnarray*}
K_i(M) &=& U_i \oplus I_{i,i+1}(M) \\
       &=& U_i \oplus U_{i+1} y \oplus I_{i,i+2}(M) \\
       &=& \ldots \\
       &=& U_i \oplus U_{i+1}y \oplus \ldots U_l y^{l-i} \oplus
           I_{i,l+1}(M) \\ 
       &=& U_i \oplus U_{i+1}y \oplus \ldots U_l y^{l-i} \oplus
           K_{i-1}(M). 
\end{eqnarray*}
Since $K_l(M) = M$ and $K_0(M) = \{0\}$ we deduce that 
$M = \oplus_{i=1}^l \oplus_{j=i}^l U_j y^{j-i}$ as $\O_K$-module, or, 
equivalently by reorganising the terms of the direct sum, 
$M = \oplus_{k=1}^l \oplus_{h=0}^{k-1} U_k y^h$.
Now (c) and (d) imply that 
$M = W_1 \oplus \dots \oplus W_l$, a direct sum of
$P_l(\O_K)$-modules.
\qedhere
\end{enumerate}
\end{proof}

\begin{remark}
{\rm 
To visualize standard $P_l(\O_K)$-modules we use diagrams
such as the following for the case $l = 4$.
\[
\begin{array}{|l|l|l|l|}
\hline
 & & & U_4 \\
\hline
 & & U_3 & U_4 y \\
\hline
 & U_2 & U_3 y & U_4 y^2 \\
\hline
U_1 & U_2 y & U_3y^2 & U_4 y^3  \\
\hline
\end{array}
\]

\noindent
The rows illustrate the series of $K_i$ in $M$ since
\[ K_i(M) = U_i \oplus U_{i+1} y \oplus \ldots \oplus U_l y^{l-i}
            \oplus K_{i-1}(M), \]
and they also exhibit its refinement by the series $(\Sigma)$. The 
columns illustrate the direct decomposition of $M$ as 
$P_l(\O_K)$-module since 
\[ W_j = U_j \oplus U_j y \oplus \ldots \oplus U_j y^{j-1}.\]
Certain submodules have multiple names.
For example, $K_4(M) = K_5(M) = I_{4,4}(M)$ and
$K_3(M) = L_4(M) = I_{3,3}(M) = I_{4,5}(M)$.
More generally, $I_{i,l+1}(M) = K_{i-1}(M) = I_{i-1,i-1}(M)$. 
}
\end{remark}

Theorem~\ref{standnf} has as an immediate consequence the 
following effective isomorphism test for standard modules. 

\begin{theorem}\cite[Lemma 12]{Gru80}
\label{standisom}
Let $M$ and $\hat{M}$ be two standard $P_l(\O_K)$-modules with 
standard generating sequences $\FF$ and $\hat{\FF}$, respectively. 
\begin{enumerate}
\item[(a)]
$M \cong \hat{M}$ as $P_l(\O_K)$-modules if and only if the types of 
$M$ and $\hat{M}$ agree.
\item[(b)]
If $M \cong \hat{M}$ as $P_l(\O_K)$-modules, then each bijection $\FF
\ra \hat{\FF}$ mapping $\FF_i$ onto $\hat{\FF}_i$ for $1 \leq i \leq l$
induces a $P_l(\O_K)$-module isomorphism from $M$ to $\hat{M}$.
\end{enumerate}
\end{theorem}

\begin{remark}\label{standard-mod-isom}
{\rm
  To decide if a finitely generated torsion-free $\O_K$-module is free,
and if so find a basis, is a classical problem from algebraic number
theory; we outline a solution in Remark \ref{steinitz}.
We decide if a given $P_l(\O_K)$-module $M$ is standard by
computing $Q_i(M) = K_i(M)/L_i(M)$ directly from the definition and then
  testing whether $Q_i(M)$ is free for $1 \leq i \leq l$; this
also yields the type.  If $M$ is standard, then,
 by pulling back the bases along the natural projections
$K_i(M) \to Q_i(M)$ for $1 \leq i \leq l$,
  we compute a standard generating sequence for $M$.
  Theorem~\ref{standisom} underpins 
an algorithm to decide if two standard $P_l(\O_K)$-modules
  are isomorphic and, if so, to construct an isomorphism.
}
\end{remark}

We now investigate the automorphism group of a standard $P_l(\O_K)$-module 
$M$ of type $(r_1, \dots, r_l)$. We use the natural action of an 
automorphism of $M$ on the quotients $Q_1(M), \dots, Q_l(M)$ to define
\[ \rho : \Aut_{P_l(\O_K)}(M) \ra 
          \Aut_{\O_K}(Q_1(M)) \times \dots \times \Aut_{\O_K}(Q_l(M)).\]

\begin{theorem} \cite[Lemma 16]{Gru80}
\label{standauto}
Let $M$ be a standard $P_l(\O_K)$-module of type $(r_1, \dots, r_l)$.
Then $\rho$ is surjective and $\Aut_{P_l(\O_K)}(M)$ splits over $\ker(\rho)$.
\end{theorem}

\begin{proof}
Let $\FF = (\FF_1, \dots, \FF_l)$ be a 
standard generating sequence for $M$ and let $U_j$ be the 
$\O_K$-module generated by $\FF_j$. So $U_j \cong Q_j(M)$ is free of rank 
$r_j$. Thus $\Aut_{\O_K}(U_j) \cong \Aut_{\O_K}(Q_j(M))$ and 
$\Aut_{\O_K}(U_j)$ acts naturally on $W_j = U_j \oplus U_j y \oplus \dots 
\oplus U_j y^{j-1}$.  This allows us to define an embedding 
$\Aut_{\O_K}(U_j) \ra 
\Aut_{P_l(\O_K)}(M)$, where $\Aut_{\O_K}(U_j)$ acts naturally on $W_j$ and 
leaves the other direct summands of $M = W_1 \oplus \dots \oplus W_l$
invariant. Hence $\rho$ is surjective and the construction allows us to 
construct a subgroup of $\Aut_{P_l(\O_K)}(M)$ isomorphic to the image 
of $\rho$. 
\end{proof}

\begin{remark}\label{standautoremark}
{\rm 
The proof of Theorem \ref{standauto} reduces the construction of generators 
for a subgroup of $\Aut_{P_l(\O_K)}(M)$ isomorphic to $\im(\rho)$ to the 
construction of generators 
for $\Aut_{\O_K}(Q_j(M))$ for $1 \leq j \leq l$.
Since $Q_j(M)$ is free, 
$\Aut_{\O_K}(Q_j(M)) \cong \GL(r_j, \O_K)$;
it is generated by $\SL(r,\O_K)$ and 
diagonal matrices with diagonal of the form
$(u,1,\dotsc,1)$ where $u$ is a unit of $\O_K$. 
Generators for $\SL(r,\O_K)$ can be determined 
from a $\Z$-basis of $\O_K$ if $r \geq 3$, or $r = 2$ and $K$ is real 
quadratic, as described in \cite{BMS1967, Vas1972}. If $r = 2$ and $K$ is 
imaginary quadratic, then generators for $\SL(r,\O_K)$ can be computed as
described in \cite{Swan1971}, see also \cite[Chapter 7]{EGM1998}; the 
authors of \cite{Braun2015} solve this case in greater generality and 
provide an implementation in \textsc{Magma}.
Generators for the unit group of $\O_K$ can be computed using the algorithms in 
\cite[Section 6]{Co1993}. 
}
\end{remark}

It remains to construct generators for $\ker(\rho)$. As before, let $M$ be a
standard module with standard generating sequence 
$\FF = (\FF_1, \dots, \FF_l)$.
Each element $\alpha$ of $\ker(\rho)$ is determined by its images on
$\FF$. Since $\alpha$ maps $K_j(M)$ to $K_j(M)$ and induces the identity 
on $Q_j(M)$, for each $f \in \FF_j$ there exists $t_f \in L_j(M)$ 
with $\alpha(f) = f+t_f$. The following lemma shows that this 
characterises the elements of $\ker(\rho)$.

\begin{lemma}
\label{charker}
Let $M$ be a standard $P_l(\O_K)$-module with standard generating
sequence $\FF$.
For each $f \in \FF_j$ choose an arbitrary $t_f \in L_j(M)$; the mapping
$\FF \ra \FF : f \ms f+t_f$ induces an element of $\ker(\rho)$.
\end{lemma}

\begin{proof}
Let $\hat{\FF}_j = \{ f + t_f \mid f \in \FF_j\}$ for $1 \leq j \leq l$.
If $\varphi_j : K_j(M) \ra Q_j(M)$ is the natural epimorphism, then 
$\varphi_j(\FF_j) = \varphi_j(\hat{\FF}_j)$. Thus $\FF_j$ and
$\hat{\FF}_j$ are both sets of preimages of the same free generating set
for $Q_j(M)$. Hence $\FF$ and $\hat{\FF} = (\hat{\FF}_1, \dots, 
\hat{\FF}_l)$ are both standard generating sequences for $M$. Theorem 
\ref{standisom} implies that the map $\FF \ra \FF : f \ms f + t_f$ 
extends to an $P_l(\O_K)$-automorphism of $M$ which acts trivially on each
quotient $Q_i(M)$.
\end{proof}

\begin{lemma}
\label{kerseries}
Let $M$ be a standard $P_l(\O_K)$-module. Each element of $\ker(\rho)$ 
induces the identity on each quotient of the series $(\Sigma)$. 
In particular, $\ker(\rho)$ is nilpotent.
\end{lemma}

\begin{proof}
Let $1 \leq i \leq j \leq l$ and let $\FF$ be a standard generating sequence
for $M$. 
Recall that $U_j$ is the $\O_K$-submodule of $M$ generated 
by $\FF_j$. As $M$ is a standard module, 
Theorem \ref{standnf} implies 
that $I_{i,j}(M) = U_j y^{j-i} \oplus I_{i,j+1}(M)$.

Let $\alpha \in \ker(\rho)$ and $f \in U_j$. By definition 
$\alpha(f) = f + t_f$ 
for some $t_f \in L_j(M)$. Thus 
$$\alpha(f y^{j-i}) = \alpha(f) y^{j-i} 
= (f + t_f)y^{j-i} = fy^{j-i} + t_fy^{j-i}.$$ 
Note that 
$t_f y^{j-i} \in 
L_j(M) y^{j-i} \leq I_{i,j+1}(M)$, so $\alpha$ induces 
the identity on $I_{i,j}(M)/ I_{i,j+1}(M)$. The result follows.
\end{proof} 

The construction of generators for $\ker(\rho)$ is delicate, requiring 
a special generating set for $L_j(M)$. 
For $1 \leq j \leq l$ we define 
\[ \TT_j = \bigcup_{k=j+1}^l \FF_k y^{k-j}
 \;\; \mbox{ and } \;\;
   \SS_j = \bigcup_{k=j}^l \FF_k y^{k-j}. \]

\begin{lemma}
\label{genslj}
Let $M$ be a standard $P_l(\O_K)$-module with standard generating sequence
$\FF$, let $B$ be a $\Z$-basis for $\O_K$, and let $1 \leq j \leq l$. 
\begin{enumerate}
\item[\rm (a)]
$\LL_j = \SS_1 \cup \dots \cup \SS_{j-1} \cup \TT_j$ generates $L_j(M)$ 
as $\O_K$-module.
\item[\rm (b)]
$\LL_j B = \{ g b \mid g \in \LL_j, b \in B \}$ generates $L_j(M)$ as
free abelian group.
\item[\rm (c)]
$\LL_j B \cap I_{i,k}(M)$ generates $L_j(M) \cap I_{i,k}(M)$ as free 
abelian group for $1 \leq i \leq l$ and $i \leq k \leq l+1$.
\end{enumerate}
\end{lemma}

\begin{proof}
\mbox{}
\begin{enumerate}
\item[
(a)] Theorem \ref{standnf} implies that for $1 \leq j \leq l$
\begin{eqnarray*}
K_j(M) & = & K_{j-1}(M) \oplus 
U_j \oplus U_{j+1} y \oplus \dots \oplus U_l y^{l-j} \\ 
L_j(M) & = & K_{j-1}(M) \oplus U_{j+1} y \oplus \dots \oplus U_l y^{l-j}.
\end{eqnarray*}
Since $\FF_i y^k$ is a set of $\O_K$-generators for $U_i y^k$ for each
$i$ and $k$, $\SS_1 \cup \dots \cup \SS_{j - 1}$ generates $K_{j-1}(M)$.
The result follows.

\item[
(b)] This follows directly from (a).

\item[
(c)] Since $L_j(M) = I_{j,j+1}(M)$, either $L_j(M) \leq 
I_{i,k}(M)$ or $I_{i,k}(M) \leq L_j(M)$. In the first case the
result follows from (b). Consider the second case. 
Since $M$ is a standard module, $I_{i,k}(M) 
= U_ky^{k-i} \oplus \dots \oplus U_ly^{l-i} \oplus K_{i-1}(M)$.
Using (a), we obtain that $I_{i,k}(M)$ is generated as $\O_K$-module by
$$G_{i,k} := \FF_k y^{k-i} \cup \dots \cup \FF_l y^{l-i} \cup \SS_1 
\cup \dots \cup \SS_{i-1}.$$ 
If $I_{i,k}(M) \leq L_j(M)$, then $G_{i,k}
\subseteq \LL_j$. Hence $\LL_j \cap I_{i,k}(M)$ generates $I_{i,k}(M)$
as $\O_K$-module and $(\LL_j \cap I_{i,k}(M)) B = \LL_j B \cap 
I_{i,k}(M)$ generates it as abelian group. \qedhere
\end{enumerate}
\end{proof}

For $f \in \FF_i$ and $g \in L_i(M)$ define $\xi_{f,g} \in \ker(\rho)$
via $\xi_{f,g}(f) = f+g$ and $\xi_{f,g}(h) = h$ for all $h \in \FF$ 
with $h \neq f$. Note that this is well-defined by Lemma \ref{charker}.

\begin{theorem}
\label{standker}
Let $M$ be a standard $P_l(\O_K)$-module with standard generating
sequence $\FF$ 
and let $B$ be a $\Z$-basis for $\O_K$. A generating set for $\ker(\rho)$ 
is 
\[ \bigcup_{i=1}^l 
   \{ \xi_{f,g} \mid f \in \FF_i, g \in \LL_i B \}.\]
\end{theorem}

\begin{proof}
Write $K_i = K_i(M)$, $L_i = L_i(M)$, $Q_i = Q_i(M)$ and $I_{i,k} = 
I_{i,k}(M)$. Recall that 
$U_i$ is the $\O_K$-submodule of $M$ generated by $\FF_i$. Let $J = 
\ker(\rho)$ and let $A = \langle \xi_{f,g} \mid 1 \leq i \leq l, 
f \in \FF_i, g \in \LL_i B \rangle$. The definition of $\xi_{f,g}$ 
implies that $A \subseteq J$. It remains to show $J \subseteq A$. 
\medskip

First we outline 
the general strategy of the proof. Let $\alpha \in J$. We use 
induction along the quotients of the series $(\Sigma)$
to determine a word in the generators of $A$ that coincides with $\alpha$. 
Note that $\alpha$ acts trivially on
$M / I_{l,l}$, the first quotient of the series.
In the induction step we assume that $\alpha$ acts 
trivially on $M/I_{i,k}$ for some $1 \leq i \leq l$ and some $i \leq k
\leq l$ and we determine $\gamma \in A$ so that $\alpha \gamma^{-1}$ 
acts trivially on $M/I_{i,k+1}$. We then replace $\alpha$ by $\alpha 
\gamma^{-1}$ and iterate this construction. Eventually, this produces 
an automorphism $\alpha$ that acts trivially on $M/I_{1,l+1} = M/\{0\} 
= M$ and thus is the identity. 
\medskip

Now we consider 
the induction step. Let $\alpha \in J$ and assume that $\alpha$ 
acts trivially on $M/I_{i,k}$ for some fixed $i$ and $k$. 
Let $M \ra M/I_{i,k+1} : m \ms \ol{m}$ denote the natural epimorphism onto 
the quotient $M/I_{i,k+1}$. 
Let $\ol{I}_{i,k}$ and $\ol{L}_j$ 
be the images of $I_{i,k}$ and $L_j$, respectively, 
under this epimorphism. 
Let
\[ H = \bigoplus_{j=1}^l \Hom_{\O_K}(U_j, \ol{L}_j \cap \ol{I}_{i,k}). \]
For $1 \leq j \leq l$ define $h_j : U_j \ra \ol{L}_j \cap \ol{I}_{i,k} : 
f \ms \ol{\alpha(f)-f}$. Note that this is well-defined, since $\alpha$
acts trivially on $M/I_{i,k}$ by assumption and the definition of $J$ asserts
that $\alpha(f)-f \in L_j$ for $f \in U_j$. Hence $h_j \in 
\Hom_{\O_K}(U_j, \ol{L}_j \cap \ol{I}_{i,k})$ for $1 \leq j \leq l$ and we
obtain the map
\[ \beta : J \ra H : \alpha \ms (h_1, \dots, h_l).\]
By Lemma \ref{kerseries} each element of $J$ acts trivially on 
$\ol{I}_{i,k}$ and 
so also on $\ol{L}_j \cap \ol{I}_{i,k}$ for $1 \leq j \leq l$. 
Thus $\beta$ is an epimorphism from the multiplicative group $J$ onto
the additive group $H$. Recall from Lemma \ref{genslj}(c) that
$\L_jB \cap I_{i,k}$ generates $L_j \cap I_{i,k}$ as abelian group. Hence 
for each $f \in \FF_j$ there exists $c_{f,g} \in \Z$ with 
\[ h_j(f) = \sum_{g \in \L_j B \cap I_{i,k}} c_{f,g} \ol{g}.\]
Let
\[ \gamma = \prod_{j=1}^l 
            \prod_{f \in \FF_j} 
            \prod_{g \in \L_j B \cap I_{i,k}}
            \xi_{f,g}^{c_{f,g}} \in A.\]
Now $\beta(\gamma) = \beta(\alpha)$. It follows that $\alpha 
\gamma^{-1}$ acts trivially on $M/I_{i,k+1}$. This completes the
induction step.
\end{proof}

\begin{remark}
\label{gruncorrected}
{\rm 
Let $\alpha \in \ker(\rho)$. For each $1 \leq j \leq l$ there exists
an $\O_K$-homomorphism $h_j : U_j \ra L_j(M) : f \ms \alpha(f) - f$. Thus
there is a bijection 
\[ \ker(\rho) \ra H = \bigoplus_{j=1}^l \Hom_{\O_K}(U_j, L_j(M)) :
    \alpha \ms (h_1, \dots h_l).\]
This bijection is not necessarily a group homomorphism. Nonetheless,
Theorem \ref{standker} essentially claims that there exists a special 
generating set for $H$ that maps to a generating set of $\ker(\rho)$ via 
this bijection. In \cite[Lemma 16 ii)]{Gru80} it is claimed that an 
arbitrary set of generators for $H$ yields a generating set for 
$\ker(\rho)$. This is not always true as the following example shows.

Let $K = \Q$ with maximal order $\O_K = \Z$ and let
$M = \Z^3$ be the $P_3(\O_K) = \Z[x, y]/(x - 1, y^3)$-module, where $x$ acts
as the identity on $M$ and $y$ acts via multiplication from the right as
\[ \left( \begin{array}{rrr}
         0 & 0 & -1 \\ 0 & 1 & 1 \\ 0 & -1 & -1
       \end{array} \right). \]
Observe that $M$ is a standard module of type $(0,0,1)$. Thus $U_1$ and
$U_2$ are trivial, $U_3 = \langle (1,0,0) \rangle$ and $L_3(M) = 
\langle (0, 1, 0), (0,0,1) \rangle$. Let $f = (1,0,0) \in U_3$;
it generates $U_3$ as $\O_K$-module and $M$ as $P_3(\O_K)$-module.
Thus $\{ f, fy, fy^2 \} = \{ (1,0,0), (0,0,-1), (0,1,1)\}$ generates 
$M$ as $\O_K$-module. Also $H = \Hom_{\O_K}(U_3, L_3(M))$. 

Let $g_1 = (0,1,0)$ and $g_2 = (0,0,1)$. Now $L_3(M) = \langle g_1,
g_2 \rangle$. Thus $H$ is generated by $\beta_1, \beta_2$ defined by 
$\beta_i : f \ms g_i$.  These homomorphisms expand to
\begin{eqnarray*}
  \beta_1 &:& f \phantom{y^2} = (1,0,\phantom{-}0) \ms g_1 \phantom{y^2}= (0,1,0) \\
        &&     f y\phantom{^2} = (0,0,-1) \ms g_1 y\phantom{^2} = (0,1,1) \\
        &&     f y^2 = (0,1,\phantom{-}1) \ms g_1 y^2 = (0,0,0); \\
  \beta_2 &:& f\phantom{y^2} = (1,0,\phantom{-}0) \ms g_2\phantom{y^2} = (0,\phantom{-}0,\phantom{-}1) \\
        &&     f y\phantom{^2} = (0,0,-1) \ms g_2 y\phantom{^2} = (0,-1,-1) \\
        &&     f y^2 = (0,1,\phantom{-}1) \ms g_2 y^2 = (0,\phantom{-}0,\phantom{-}0).
\end{eqnarray*}
Let $\alpha_i = \beta_i + \id$. Note that $g_1 = fy + fy^2$ and $g_2 = 
- fy$. Thus
\begin{eqnarray*}
\alpha_2(\alpha_1(f)) 
  & = & \alpha_2( f + g_1 ) 
  = \alpha_2( f + fy + fy^2 ) 
  =  \alpha_2( f ) + \alpha_2( fy ) + \alpha_2( fy^2 ) \\
  & = & (f-fy) + (f-fy)y + (f-fy)y^2
  = f.
\end{eqnarray*}
Hence $\alpha_1 = \alpha_2^{-1}$ and $\langle \alpha_1, \alpha_2 \rangle$ 
is infinite cyclic. 

Let $\bar{g}_1 = fy = (0,0,-1)$ and $\bar{g}_2 = fy^2 = (0,1,1)$. Now $L_3(M) =
\langle \bar{g}_1, \bar{g}_2 \rangle$ and this generating set corresponds 
to that chosen in Theorem \ref{standker}. 
As before, $H$ is generated by $\gamma_1, \gamma_2$ defined 
by $\gamma_i: f \ms \bar{g}_i$. 
Further, $\gamma_1$ maps $f \ms fy$, $fy \ms fy^2$ and $fy^2 \ms fy^3=0$
and $\gamma_2$ maps $f \ms fy^2$, $fy \ms 0$ and $fy^2 \ms 0$. Let
$\alpha_i = \gamma_i + \id$. Now $\langle \alpha_1, \alpha_2 \rangle$
is free abelian of rank $2$ and corresponds to $\ker(\rho)$.
}
\end{remark}

\begin{remark}\label{autgenremark}
{\rm
We use  Remark~\ref{standautoremark} and Theorem~\ref{standker} 
to obtain a finite generating set for the automorphism group
of a standard $P_l(\O_K)$-module. 
}
\end{remark}

\subsection{The construction of standard submodules}

We now introduce a method to construct standard submodules of 
a given finite index in an integral $P_l(\O_K)$-module $M$. Our approach 
and its effectiveness contrast with that of~\cite{Gru80}. 
Since $M$ is integral, $Q_i(M) = K_i(M)/L_i(M)$ is a finitely generated abelian group and so
has a torsion subgroup; we denote this by $\hat{L}_i(M)/L_i(M)$ and its
associated torsion-free quotient by $\hat{Q}_i(M) = K_i(M) / \hat{L}_i(M)$.
Note that $Q_i(M)$ is an $\O_K$-module and $\hat{Q}_i(M)$ is an integral 
$\O_K$-module.

\subsubsection{Free submodules of finite index}
\label{freesubs}

Fix $i \in \{1, \dots, l\}$ and write $Q = Q_i(M)$ and $\hat{Q} = 
\hat{Q}_i(M)$. Note that $Q$ is a finitely generated abelian group and
an $\O_K$-module. If $T(Q)$ is the torsion subgroup of $Q$, then 
$\hat{Q} = Q / T(Q)$ is an integral $\O_K$-module. 

\begin{remark}
\label{steinitz}
{\rm 
Steinitz theory shows that there exists a free $\O_K$-submodule in 
$\hat{Q}$ of finite index. To construct such, use \cite[Theorem 1.2.19]{Co2000}
to decompose $\hat{Q} = F \oplus I$ as $\O_K$-module, where $F$ is a free 
$\O_K$-module and $I$ can be identified with an ideal in $\O_K$. If $I$ 
is principal, then $\hat{Q}$ is free. 
Otherwise, consider
the natural homomorphism $\varphi$ from the set of non-zero ideals of 
$\O_K$ onto the ideal class group 
of $K$ and let $J$ be an ideal
in $\O_K$ with $\varphi(J) = \varphi(I)^{-1}$. Now $IJ$ is a non-zero
principal ideal of $\O_K$ and $F \oplus IJ$ is a free $\O_K$-submodule
of finite index in $\hat{Q}$. Algorithms to compute $F, I, J$ and $IJ$ are
described in \cite[Section 6.5]{Co1993} and \cite[Chapter 1]{Co2000}. 
By choosing $J$ to have minimal norm, we find a free $\O_K$-submodule 
of minimal index.
}
\end{remark}

Note that there are always only finitely many $\O_K$-submodules of a 
given finite index in $\hat{Q}$, as $\hat{Q}$ is a free abelian group of finite
rank. 

\begin{lemma}
\label{freeexist}
Let $t = |T(Q)|$ and let $r = \rank(Q) / \rank(\O_K)$.
\begin{items}
\item[\rm (a)]
Suppose that $Q$ has a free $\O_K$-submodule of index $w$. Then $t \mid w$ 
and $\hat{Q}$ has a free $\O_K$-submodule of index $w/t$.
\item[\rm (b)]
Suppose that $\hat{Q}$ has $u$ free $\O_K$-submodules of index $v$. Then 
$Q$ has $u t^r$ free $\O_K$-submodules of index $vt$.
\end{items}
\end{lemma}

\begin{proof}
\mbox{}
\begin{enumerate}
\item[(a)]
Let $F$ be a free submodule of $Q$ of index $w$, so $F$ is torsion free
as abelian group and hence $F \cap T(Q) = \{0\}$. Thus $t \mid w$ and
$F + T(Q) / T(Q)$ is a free submodule of index $w/t$ in $\hat{Q}$.
\item[(b)]
Let $F/T(Q)$ be a free submodule of index $v$ in $\hat{Q}$, so $F$
splits over $T(Q)$ and every complement to $T(Q)$ in $F$ is a free 
submodule of $Q$. Such a complement has index $vt$ and there are
$t^r$ such complements, as $F/T(Q)$ is free of rank $r$.
\qedhere
\end{enumerate}
\end{proof}

\begin{remark}\label{allfree}
{\rm
Later we must determine all submodules
of fixed index $m$ in an integral $\O_K$-submodule of $M$.
  We use the Chinese remainder theorem to restrict to the case where
  $m$ is a prime power $p^e$.
So we induce the action of $\O_K$ on an
integral $\O_K$-submodule of $M$ to an action on $(\Z/p\Z)^n$
  and determine the $\O_K$-submodules of bounded codimension.
  The submodules of bounded codimension in a module defined over
a finite field can be effectively
  computed using variants of the {\sc MeatAxe}; for more details see
  \cite[Chapter 7]{HEO05}.
  By eliminating all non-free $\O_K$-submodules, we find all
free $\O_K$-submodules of an integral $\O_K$-module with a fixed index.
}

\end{remark}

\subsubsection{Standard submodules of finite index}
\label{standfind}

We now discuss how to construct the standard submodules of finite index
in an integral $P_l(\O_K)$-module $M$ of type $(r_1, \dots, r_l)$. 

\begin{lemma}
\label{finsub}
Let $M$ be an integral $P_l(\O_K)$-module and let $S$ be a submodule of finite 
index in $M$. Now $K_i(S) = K_i(M) \cap S$ and $L_i(S)$ is a subgroup of 
finite index in $L_i(M) \cap S$ for $1 \leq i \leq l$. Moreover $S$ has the
same type as $M$.
\end{lemma}

\begin{proof}
Let $1 \leq i \leq l$. Clearly, $K_i(S) = \{ s \in S \mid s y^i = 0\}
= K_i(M) \cap S$. Hence $$L_i(S) = K_{i+1}(S)y + K_{i-1}(S) \leq 
K_{i+1}(M) y + K_{i-1}(M) = L_i(M),$$ so $L_i(S) \leq L_i(M) \cap S = 
L_i(M) \cap K_i(M) \cap S = L_i(M) \cap K_i(S)$. 

Define $A_i = K_i(S) / (K_i(S) \cap L_i(M))$ and $B_i=(K_i(S) + L_i(M))/L_i(M)$.
Then $A_i \cong B_i$ via the natural homomorphism, $A_i$ is isomorphic to 
a quotient of $Q_i(S)$ and $B_i$ is a submodule of $Q_i(M)$. 
The finite index of $S$ in $M$ implies that $K_i(S)$ has finite index 
in $K_i(M)$. Thus $B_i$ has finite index in $Q_i(M)$. 
We illustrate these relationships in Figure \ref{lattice}.

\begin{figure}[htb]
\begin{center}
\setlength{\unitlength}{1.5cm}
\begin{picture}(5,5)
\thicklines
\put(0,0){\circle*{0.1}}
\put(0,1){\circle*{0.1}}
\put(0,2){\circle*{0.1}}
\put(0,0){\line(0,1){2}}

\put(1,0){\circle*{0.1}}
\put(1,1){\circle*{0.1}}
\put(1,2){\circle*{0.1}}
\put(1,0){\line(0,1){2}}

\put(2,2){\circle*{0.1}}
\put(2,3){\circle*{0.1}}
\put(2,4){\circle*{0.1}}
\put(2,2){\line(0,1){2}}

\put(3,2){\circle*{0.1}}
\put(3,3){\circle*{0.1}}
\put(3,4){\circle*{0.1}}
\put(3,2){\line(0,1){2}}

\put(1,1){\line(1,1){1}}
\put(1,2){\line(1,1){1}}
\put(2.3,3){\vector(1,0){0,3}}
\put(0.7,1){\vector(-1,0){0,3}}

\put(-0.7,2){$Q_i(S)$}
\put(3.1,4){$Q_i(M)$}

\put(0.39,2.1){$K_i(S)$}
\put(1.1,0.0){$L_i(S)$}

\put(2,1.8){$L_i(M)$}
\put(2.05,4){$K_i(M)$}

\put(3.1,2.4){$B_i$}
\put(0.1,1.5){$A_i$}

\end{picture}
\end{center}
\caption{Illustrating the relationships}\label{lattice}
\end{figure}
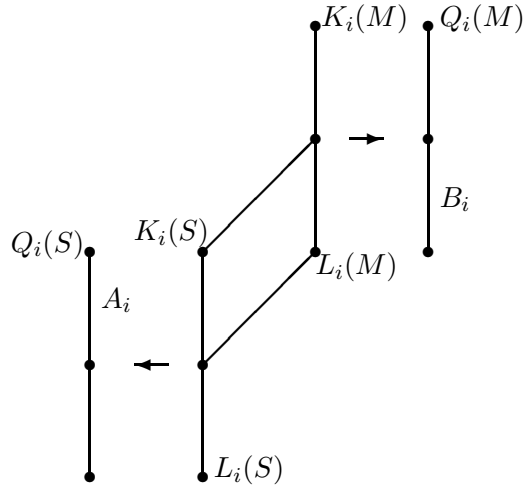

It remains to show that $L_i(S)$ has finite index in $L_i(M) \cap K_i(S)$
and that the types of $S$ and $M$ agree. Recall that $\rank$ denotes the 
torsion free rank of an abelian group. Using Lemma \ref{modseries}, we 
deduce that 

\begin{eqnarray*}
\rank(S) 
   &=& \sum_{i=1}^l \sum_{j=i}^l \rank(Q_j(S)) \\
   &\geq& \sum_{i=1}^l \sum_{j=i}^l \rank(A_j) \\
   &=& \sum_{i=1}^l \sum_{j=i}^l \rank(B_j) \\
   &=& \sum_{i=1}^l \sum_{j=i}^l \rank(Q_j(M)) \\
   &=& \rank(M).
\end{eqnarray*}
Since $S$ has finite index in $M$, $\rank(S) = \rank(M)$. Hence  
$\rank(Q_i(S)) = \rank(A_i)$ for $1 \leq i \leq l $.
Thus $L_i(S)$ has finite index in $L_i(M) \cap K_i(S)=L_i(M) \cap S$. 
It also implies that $\rank(Q_j(S)) = \rank(Q_j(M))$ for $1 \leq j \leq l$, 
so the types of $S$ and $M$ agree.
\end{proof}

We next observe that each standard submodule is associated with
certain free $\O_K$-submodules of $Q_1(M), \dots, Q_l(M)$.

\begin{theorem}
\label{goodexist}
Let $M$ be an integral $P_l(\O_K)$-module and let $S$ be a standard 
submodule of $M$ of finite index.
\begin{items}
\item[\rm (a)]
Let $1 \leq i \leq l$. Then $L_i(S) = L_i(M) \cap S$ and 
$$\rho_i \colon Q_i(S) \ra Q_i(M) : a+L_i(S) \ms a + L_i(M)$$
 is injective
with image $R_i(S)$ of finite index in $Q_i(M)$.
\item[\rm (b)]
Write $[Q_i(M) : R_i(S)] = h_i$ for $1 \leq i \leq l$. Then
\[ [M:S] = \prod_{i=1}^l \prod_{j=i}^l h_j = \prod_{k=1}^l h_k^k.\]
\end{items}
\end{theorem}

\begin{proof}
\begin{items}
\item[\rm (a)]
By Lemma \ref{finsub}, $L_i(S)$ is a submodule of finite index in
$L_i(M) \cap K_i(S) = L_i(M) \cap S$. Thus $(L_i(M) \cap K_i(S))/L_i(S)$
is a finite abelian subgroup of $Q_i(S)$. Since $S$ is standard, 
$Q_i(S)$ is free abelian so $L_i(M) \cap S = L_i(S)$. 
Hence $\rho_i$ is a monomorphism. Its image $R_i(S)$ has finite
index in $Q_i(M)$, as $S$ has finite index in $M$.
\item[\rm (b)]
We use the series $(\Sigma)$ and Lemma \ref{modseries}. Consider the
epimorphism $\sigma_{i,j} \colon K_j(M) \ra I_{i,j}(M)/I_{i,j+1}(M)$ with
kernel $L_j(M)$. Now $\sigma_{i,j}(K_j(S))$ has kernel $L_j(M) \cap S
  = L_j(S)$; thus the image of $\sigma_{i,j}$ restricted to $K_j(S)$ is
isomorphic to $R_j(S)$ and has index $h_j$ in the image of 
$\sigma_{i,j}$. Computing the index by factorising it through 
$(\Sigma)$ now yields the result. \qedhere 
\end{items}
\end{proof}

\begin{remark}\label{minimalstandard}
{\rm
Theorem \ref{goodexist} asserts that every standard submodule of the
integral $P_l(\O_K)$-module $M$ is associated with a sequence 
$(R_1, \dots, R_l)$ 
where $R_i$ is a free $\O_K$-submodule of finite index
in $Q_i(M)$ for $1 \leq i \leq l$.
We call the integers $(h_1, \ldots, h_l)$ from Theorem~\ref{goodexist}~(b), 
which are invariant under isomorphism, the {\em indices 
associated} with $S$.
Vice versa, by lifting bases along the natural projections $K_i(M) \to Q_i(M)$, each such sequence $(R_1,\dotsc,R_l)$ of free submodules with finite index determines a standard
submodule of $M$ with finite index.
We use Theorem~\ref{goodexist} and Remark~\ref{steinitz}
to determine standard submodules of $M$ with minimal index
by computing free submodules of minimal
index of $Q_i(M)$ for $1 \leq i \leq l$.
}
\end{remark}

We now determine the number of
standard submodules associated with $(R_1, \dots, R_l)$. 

\begin{theorem}
\label{goodcount}
Let $M$ be an integral $P_l(\O_K)$-module of  type $(r_1, \ldots, r_l)$. 
For $1 \leq i \leq l$ let 
$R_i$ be a free $\O_K$-submodule of finite index $h_i$ in $Q_i(M)$.
Define
\[ \pi_i = \frac{(h_i \cdots h_l)^{r_i+ \dots + r_l}}{h_i^{r_i}}.\]
 There are exactly $\pi_1 \cdots \pi_l$ standard submodules associated 
with $(R_1, \dots, R_l)$ in $M$.
\end{theorem}

\begin{proof}
For $1 \leq i \leq l$ let $\varphi_i \colon K_i(M) \ra Q_i(M)$ be the 
natural epimorphism. Let $\FF_i$ be an arbitrary set of preimages under
$\varphi_i$ of a free generating set of the free $\O_K$-module $R_i \leq 
Q_i(M)$. Since $R_i$ is free, $\FF_i$ generates a free
$\O_K$-submodule $U_i$ of $M$. As in Theorem \ref{standnf}, it
now follows that the $P_l(\O_K)$-module $S$ generated by $\FF_1 \cup \dots 
\cup \FF_l$ is standard. It has finite index in $M$ by Theorem 
\ref{goodexist}(b).

Every standard submodule of finite index in $M$ and associated with 
$(R_1, \dots, R_l)$ arises via this construction.
Hence the number of such submodules depends on the number 
of essentially different ways 
to choose $(\FF_1, \dots, \FF_l)$. Choosing a different free generating
set for $R_i$ does not affect the submodule, but choosing a different set
of preimages under $\varphi_i$ may change the submodule. 

We use induction along the series $M = K_l(M) \geq \ldots \geq K_0(M) = \{0\}$
to count the different possibilities for standard submodules. For the initial
step of this induction, we note that $M/K_{l-1}(M) \cong Q_l(M)$ has only one 
standard submodule associated with $(R_l)$, namely $R_l$ itself. This 
accords with the stated result, since
\[ \pi_l = \frac{h_l^{r_l}}{h_l^{r_l}} = 1.\]

For the induction step we consider $M/K_1(M)$. First note that $M/K_1(M)$ is
an integral $P_l(\O_K)$-module of type $(r_2, \ldots, r_l)$ and it satisfies 
$Q_{i-1}(M/K_1(M)) \cong Q_i(M)$ 
for $2 \leq i \leq l$. We assume by induction that the number of distinct 
standard submodules associated with $(R_2, \ldots, R_l)$ in $M/K_1(M)$ is 
given by $\pi_2 \cdots \pi_l$.
We show that for every standard submodule $S/K_1(M)$ in $M/K_1(M)$ 
associated with $(R_2, \ldots, R_l)$ there exists $\pi_1$ standard 
submodules $T$ in $M$ associated with $(R_1, \ldots, R_l)$ such that 
$T+K_1(M) = S$.  Our proof of the induction step has three parts. 
\medskip

{\bf Part 1.}
We determine the number of possible options for $T + L_1(M)$, where $T$ is 
a standard submodule of $M$ associated with $(R_1, \ldots, R_l)$ such that 
$T+K_1(M) = S$. 

Let $R$ be a full preimage of $R_1$ under 
$\varphi_1 : K_1(M) \ra Q_1(M)$. Then $(T + L_1(M)) \cap
K_1(M) = R$ and $(T + L_1(M)) + K_1(M) = T + K_1(M) = S$. Thus the 
desired submodules $T + L_1(M)$ correspond to the complements in $S$ to the 
section $K_1(M)/R$ and hence, in turn, to $\Hom_{P_l(\O_K)} (S/K_1(M), 
K_1(M)/R)$. Since $S/K_1(M)$ is generated as $P_l(\O_K)$-module by 
$r_2 + \ldots + r_l$ generators, and $[K_1(M):R] = [Q_1(M): R_1] = h_1$, 
the number of such complements is $h_1^{r_2+\ldots+r_l}$.
\medskip

{\bf Part 2.} Let $T$ and $\hat{T}$ be two standard submodules of $M$ 
associated with $(R_1, \ldots, R_l)$ satisfying $T+L_1(M) = 
\hat{T}+L_1(M)$. We show that $T \cap L_1(M) = \hat{T} \cap L_1(M)$. 

Let $(\FF_1, \ldots, \FF_l)$ and $(\hat{\FF}_1, \ldots, \hat{\FF}_l)$ be 
standard generating sequences for $T$ and $\hat{T}$, respectively. Let
$U_i$ and $\hat{U}_i$ be the $\O_K$-submodules generated by $\FF_i$ and 
$\hat{\FF}_i$, respectively. 
Theorem \ref{goodexist} asserts that $T \cap L_1(M) = L_1(T)$ and 
$\hat{T} \cap L_1(M) = L_1(\hat{T})$. Theorem \ref{standnf} yields 
that $L_1(T) = U_2 y \oplus \ldots \oplus U_l y^{l-1}$ and 
$L_1(\hat{T}) = \hat{U}_2 y \oplus \ldots \oplus \hat{U}_l y^{l-1}$. 
As $T + L_1(M) = \hat{T}+L_1(M)$, we can assume that for each 
$f \in U_i$ there exists $t_f \in L_1(M)$ so that $f+t_f \in \hat{U}_i$. 
Since $(f+t_f)y= fy + t_fy = fy$, it follows that $L_1(T) = L_1(\hat{T})$. 
\medskip

{\bf Part 3.}
Building on 
Part 1, we consider a fixed complement $U$ to $K_1(M)/R$ in $S$. 
We count the number of standard submodules $T$ in $M$ associated with 
$(R_1, \ldots, R_l)$ so that $T + L_1(M) = U$. 

Part 2 shows the intersection 
$T \cap L_1(M)$ depends on $U$, but not on $T$. Thus the desired submodules
$T$ correspond to the complements in $U$ to the section $L_1(M)/ L_1(T)$ 
and so, in turn, to $\Hom_{P_l(\O_K)}(U/L_1(M), L_1(M)/L_1(T))$. Since 
$U/L_1(M)$ is generated as $P_l(\O_K)$-module by $r_1 + \ldots + r_l$ 
generators, and $[L_1(M) : L_1(T)] = h_2 \cdots h_l$ by Lemma \ref{modseries}, 
the number of such complements is $(h_2 \cdots h_l)^{r_1 + \ldots + r_l}$.
\medskip

Finally, we combine the results of Parts 1 and 3 and the observation that 
\[ h_1^{r_2 + \dots + r_l} \cdot (h_2 \cdots h_l)^{r_1+ \dots + r_l}
  = \frac{(h_1 \cdots h_l)^{r_1+ \dots + r_l}}{h_1^{r_1}} 
  = \pi_1 \]
to deduce the claim.
\end{proof}

The proof of Theorem \ref{goodcount} is constructive and 
allows us to determine the set of all standard submodules of $M$ associated
with a given collection of free submodules $(R_1, \ldots, R_l)$.

\subsection{Isomorphisms and automorphisms for integral $P_l(\O_K)$-modules}

We now solve the isomorphism and automorphism group problems for integral 
$P_l(\O_K)$-modules. See \cite[Lemmas 15 and 19]{Gru80}.

\begin{theorem}
\label{genisom}
Let $M$ and $\hat{M}$ be two integral $P_l(\O_K)$-modules. Let $N$ be a
  standard submodule in $M$ with associated indices $(h_1,\dotsc,h_l)$ and let $\{ \hat{N}_1, \dots,
\hat{N}_s\}$ be the set of all standard submodules in
  $\hat{M}$ with associated indices $(h_1,\dotsc,h_l)$. Let $c \in \N$ be such that $cM \leq N$. 
Now $M \cong \hat{M}$
as $P_l(\O_K)$-modules if and only if there exists $j \in \{1, \dots, s\}$
and a $P_l(\O_K)$-module isomorphism $\kappa \colon N \ra \hat{N}_j$ with
$\kappa(cM) = c\hat{M}$.
\end{theorem}

\begin{proof}
Let $\gamma : M \ra \hat{M}$ be a $P_l(\O_K)$-module isomorphism. 
Now $\gamma(cM) = c \gamma(M) = c \hat{M}$ and $\gamma(N)$ is a standard
  submodule of $\hat{M}$ with associated indices $(h_1,\dotsc,h_l)$. Thus $\gamma(N) = \hat{N}_j$ for some $j \in
\{1, \dots, s\}$.

Conversely, 
let $\kappa : N \ra \hat{N}_j$ be a $P_l(\O_K)$-module isomorphism
with $\kappa(cM) = c\hat{M}$. Define 
$\gamma : M \ra \hat{M} : w \ms c^{-1} \kappa(cw)$. 
Clearly $\gamma$ is a $P_l(\O_K)$-module isomorphism.
\end{proof}

Theorem~\ref{genisom} reduces the isomorphism problem for integral 
$P_l(\O_K)$-modules
to the construction of generators for $A = \Aut_{P_l(\O_K)}(N)$ and 
the determination
of arbitrary $P_l(\O_K)$-module isomorphisms 
$\delta_j \colon N \to \hat{N_j}$, $j \in \{1,\dotsc, s\}$.
If $A$ and $\delta_j$ are given, then the orbit
$O_j = \{ \alpha(cM) \mid \alpha \in A\}$ can be constructed.
Now $M$ and $\hat{M}$ are isomorphic if and only if there exists 
$j \in \{1,\dotsc,s\}$ such that $c\hat{M}$ is contained in
$\{ \delta_j(o) \mid o \in O_j\}$.

Next, we construct $\Aut_{P_l(\O_K)}(M)$ for an
integral $P_l(\O_K)$-module $M$. Let $(h_1,\dotsc,h_l) \in \N^l$ be such that the set
$\{N_1, \dots, N_s\}$ of standard submodules of $M$ with associated indices $(h_1,\dotsc,h_l)$ is
not empty. Let $c \in \N$ so that $cM \leq N_1$.
Let $S$ be the largest subset of $\{2, \dots, s\}$ having the property
that for each $i \in S$ there exists a $P_l(\O_K)$-module isomorphism 
$\epsilon_i: N_1 \ra N_i$ with $\epsilon_i(cM) = cM$. 
Theorem~\ref{genisom} implies that each 
$\epsilon_i$ extends to an element (which we also call $\epsilon_i$)
of $\Aut_{P_l(\O_K)}(M)$. 
This allows us to define
\[ \Pi = \langle \epsilon_i \mid i \in S \rangle \leq
         \Aut_{P_l(\O_K)}(M).\]
Let $\MySigma$ denote a generating set for the stabiliser of $c M$
in $A := \Aut_{P_l(\O_K)}(N_1)$.  Theorem~\ref{genisom} implies that 
each element of $\MySigma$ extends to an element of $\Aut_{P_l(\O_K)}(M)$. 
Hence we can also consider $\MySigma$ as a subset of $\Aut_{P_l(\O_K)}(M)$.

\begin{theorem}
\label{genauto}
If $M$ is an integral $P_l(\O_K)$-module, then $\Aut_{P_l(\O_K)}(M) =
\langle \Pi, \MySigma \rangle$.
\end{theorem}

\begin{proof}
If $\alpha \in \Aut_{P_l(\O_K)}(M)$, then $\alpha(N_1) = N_j$ for some
$j$. If $j = 1$, then $\alpha \in \langle \MySigma \rangle$. If $j \neq
1$, then $\alpha \circ \epsilon_j^{-1} \in \langle \MySigma \rangle$. In
summary, $\alpha \in \langle \Pi, \MySigma \rangle$.
\end{proof}

Theorem~\ref{genauto} thus reduces the problem of constructing generators 
for $\Aut_{P_l(\O_K)}(M)$ to
the computation of standard submodules, isomorphisms between them, and
generators $\MySigma$ for $\Stab_A(cM)$.
The latter are obtained from $A$ via Schreier generators.

\section{Algorithms for modules over truncated polynomial rings}
\label{module-algorithms}

\subsection{Constructing standard submodules}
\label{algostand}

Let $M$ be an integral $P_l(\O_K)$-module. We summarise algorithms 
to construct one or all standard submodules of $M$ with fixed associated 
indices.
See Section \ref{standfind} for theoretical background. 
We denote the torsion 
subgroup of $Q_i(M)$ by $T_i(M)$ and its quotient by $\hat{Q}_i(M) =
Q_i(M)/T_i(M)$. Let $\phi_i : Q_i(M) \ra \hat{Q}_i(M)$ be the natural 
epimorphism and recall that $r_i$ is the rank of 
$\hat{Q}_i(M)$ as $\O_K$-module.

\mbox{} \\
{\bf Algorithm I.1:}\\
Let $M$ be an integral $P_l(\O_K)$-module and $(h_1, \ldots, h_l)$ a 
sequence of natural numbers. Determine all sequences $(R_1, \ldots, R_l)$, 
where $R_i$ is a free $\O_K$-submodule of index $h_i$ in $Q_i(M)$.
\begin{items}
\item[$\bullet$]
For $i \in \{1, \ldots, l\}$ do:
\begin{items}
\item[-]
Let $t_i = |T_i(M)|$. If $t_i \nmid h_i$, then return the empty set.
\item[-]
Initialise $\SS_i$ as the empty list.
\item[-]
Construct the set $\R$ of all free $\O_K$-submodules of index $h_i/t_i$ in
$\hat{Q}_i(M)$ using Remark~\ref{allfree}. 
\item[-]
For each $R \in \R$ with $\O_K$-module basis $B = (b_1, \ldots, b_{r_i})$, 
say, determine the $\O_K$-modules generated by $(b_1 + c_1, \ldots, 
b_{r_i} + c_{r_i})$ for all sequences $(c_1, \ldots, c_{r_i})$ with
entries in $T_i(M)$ and append these to $\SS_i$.
\end{items}
\item[$\bullet$]
Return the set of sequences $(R_1, \ldots, R_l)$ with $R_i \in \SS_i$ 
for $1 \leq i \leq l$.
\end{items}

\smallskip
\noindent
We nox fix a sequence $(R_1, \ldots, R_l)$ 
where $R_i$ is a free $\O_K$-submodule of finite index in $Q_i(M)$ 
for $1 \leq i \leq l$ and determine all standard submodules 
associated with this sequence.


\mbox{} \\
{\bf Algorithm I.2:} \\
Let $M$ be an integral $P_l(\O_K)$-module and for $1 \leq i \leq l$ let 
$R_i$ be a free $\O_K$-submodule of $Q_i(M)$. Determine all standard 
submodules of $M$ associated with $(R_1, \dots, R_l)$.
\begin{items}
\item[$\bullet$]
Find one standard submodule $L$ in $M$ associated with $(R_1, \dots, R_l)$ 
using Remark~\ref{minimalstandard}. 
\item[$\bullet$]
Initialise $\L$ as the list consisting of $L$.
\item[$\bullet$]
For $i$ in the list $(l, \dots, 1)$ do:
\begin{items}
\item[-]
Let $R$ be the full preimage of $R_i$ under $\varphi_i : K_i(M) \ra Q_i(M)$ and $T$ a transversal of $R$ in $K_i(M)$.
\item[-]
Initialise $\SS$ as the empty list.
\item[-]
For $S$ in $\L$ do:
\begin{items}
\item[--]
Let $(\FF_1, \ldots, \FF_l)$ be a standard generating sequence of $S$.
\item[--]
For $(t_{i+1}, \ldots, t_l) \in T^{r_{i+1}} \times \dots \times T^{r_l}$ compute the $P_l(\O_K)$-module generated by 
$(\FF_1, \ldots, \FF_i, \FF_{i+1}+t_{i+1}, \ldots, \FF_l + t_l)$ and add it to $\SS$.
\end{items}
\item[-]
Replace $\L$ by $\SS$.
\end{items}
\item[$\bullet$]
For $i$ in the list $(l, \dots, 1)$ do:
\begin{items}
\item[-]
Initialise $\SS$ as the empty list.
\item[-]
For $S$ in $\L$ do:
\begin{items}
\item[--]
Let $T$ be a transversal of $L_i(S)+K_{i-1}(M)$ in $L_i(M)$.
\item[--]
Let $(\FF_1, \ldots, \FF_l)$ be a standard generating sequence of $S$.
\item[--]
For $(t_i, \ldots, t_l) \in T^{r_i} \times \dots \times T^{r_l}$ do compute the $P_l(\O_K)$-module generated by 
$(\FF_1, \ldots, \FF_{i-1}, \FF_i+t_i, \ldots, \FF_l + t_l)$ and add it to $\SS$.
\end{items}
\item[-]
Replace $\L$ by $\SS$.
\end{items}
\item[$\bullet$]
Return $\L$.
\end{items}

\noindent
Recall that each standard submodule $S$ of finite index in an integral $P_l(\O_K)$-module
$M$ induces a sequence $(R_1, \ldots, R_l)$ with $R_i$ a free $\O_K$-submodule
of $Q_i(M)$ for $1 \leq i \leq l$. The sequence $(h_1, \ldots, h_l)$
defined by $h_i = [Q_i(M) : R_i]$ are the associated indices of $S$ in $M$.

\mbox{} \\
{\bf Algorithm I.3:} \\
Let $M$ be an integral $P_l(\O_K)$-module and $(h_1, \ldots, h_l)$ a
sequence of natural numbers. Determine all standard submodules with
associated indices $(h_1, \ldots, h_l)$ in $M$.
\begin{items}
\item[$\bullet$]
Compute all sequences $(R_1, \ldots, R_l)$ where $R_i$ is a free 
$\O_K$-submodule of index $h_i$ in $Q_i(M)$ using Algorithm I.1.
\item[$\bullet$]
For each sequence $(R_1, \ldots, R_l)$ compute all associated standard 
submodules of $M$ using Algorithm I.2.
\item[$\bullet$]
Return the full list of all computed standard submodules.
\end{items}

\subsection{Isomorphisms and automorphisms of integral $P_l(\O_K)$-modules}

We summarise algorithms to solve the isomorphism and
automorphism group problems for integral $P_l(\O_K)$-modules.

\mbox{} \\
{\bf Algorithm II.1:} \\
Let $M$ and $\hat{M}$ be two integral $P_l(\O_K)$-modules.
Decide if $M$ and $\hat{M}$ are isomorphic as $P_l(\O_K)$-modules and if
so, then determine an isomorphism.
\begin{items}
\item[$\bullet$]
If the types of $M$ and $\hat M$ differ, then return false.
\item[$\bullet$]
  Determine a standard submodule $N$ of minimal index using 
Remark~\ref{minimalstandard}. 
\item[$\bullet$]
Choose $c \in \N$ such that $cM \leq N$ and determine the indices
$(h_1, \ldots, h_l)$ associated with $N$. 
\item[$\bullet$]
Determine all standard submodules $\hat{N}_1, \dots, \hat{N}_s$ with 
associated indices $(h_1, \ldots, h_l)$ in $\hat{M}$ using Algorithm~I.3.
\item[$\bullet$]
Compute generators for $\Aut_{P_l(\O_K)}(N)$ using Remark~\ref{autgenremark}.
\item[$\bullet$]
Compute the orbit $O$ of $cM$ under $\Aut_{P_l(\O_K)}(N)$; 
since $O$ is finite it can be listed explicitly. 
Determine a transversal $T$ for $O$.
\item[$\bullet$]
Consider each $i \in \{1, \dots, s\}$ in turn:
\begin{items}
\item[-]
Use Remark \ref{standard-mod-isom} to find a $P_l(\O_K)$-isomorphism
$\gamma : N \ra \hat{N}_i$.
\item[-]
Determine $\{ \gamma(w) \mid w \in O \}$ and decide if $c\hat{M}$ is 
contained in this set. 
\item[-]
If not, then consider next $i$.
\item[-]
Take $\tau \in T$ corresponding to $w \in O$ with
$\gamma(w) = c \hat{M}$. Note that $\gamma(w) = \gamma(\tau(cM))$.
\item[-]
Lift $\gamma \tau : c M \ra c \hat{M}$ to $\sigma : M \ra \hat{M}$ using
division by $c$.
\item[-]
Return $\sigma$.
\end{items}
\item[$\bullet$]
Return fail. (There is no isomorphism $M \ra \hat{M}$.)
\end{items}

\mbox{} \\
{\bf Algorithm II.2:} \\
Let $M$ be an integral $P_l(\O_K)$-module.
Determine generators for $\Aut_{P_l(\O_K)}(M)$.
\begin{items}
\item[$\bullet$]
Determine a standard submodule $N$ of minimal index using 
Remark~\ref{minimalstandard}. 
\item[$\bullet$]
Choose $c \in \N$ such that $cM \leq N$ and determine the indices
$(h_1, \ldots, h_l)$ associated with $N$. 
\item[$\bullet$]
Determine all standard submodules $N_1, \dots, N_s$ with associated indices
$(h_1, \ldots, h_l)$ in $M$ using Algorithm~I.3. Assume $N = N_1$.
\item[$\bullet$]
Compute generators for $\Aut_{P_l(\O_K)}(N)$ using Remark~\ref{autgenremark}.
\item[$\bullet$]
Compute the orbit $O$ of $cM$ under $\Aut_{P_l(\O_K)}(N)$; 
since $O$ is finite it can be listed explicitly. 
Determine a transversal $T$ for $O$.
\item[$\bullet$]
Consider each $i \in \{2, \dots, s\}$ in turn:
\begin{items}
\item[-]
Use Remark \ref{standard-mod-isom} to determine a $P_l(\O_K)$-isomorphism
$\gamma : N \ra N_i$.
\item[-]
Determine $\{ \gamma(w) \mid w \in O \}$ and decide if $cM$ is contained 
in this set. 
\item[-]
It not, then consider next $i$.
\item[-]
Take $\tau \in T$ corresponding to $w \in O$ with $\gamma(w) = c M$
and let $\epsilon_i = \gamma \tau$.
\end{items}
\item[$\bullet$]
Let $\Pi$ be the list of determined isomorphisms $\epsilon_i$.
\item[$\bullet$]
Compute generators $\MySigma$ for the stabiliser of $cM$ in
$\Aut_{P_l(\O_K)} (N)$.
\item[$\bullet$]
Modify the elements of $\Pi$ and $\MySigma$: determine their restriction to
$cM$ and then lift this to an isomorphism $M \ra M$ using division by $c$.
\item[$\bullet$]
Return the generating set $\Pi \cup \MySigma$ for $\Aut_{P_l(\O_K)}(M)$.
\end{items}

\subsection{Isomorphisms and automorphisms of integral $P_l(\O)$-modules}

Let $M$ and $\hat M$ be two integral $P_l(\O)$-modules. Using
Theorem~\ref{reduct} and the previous section, we summarise algorithms 
to decide if $M\cong \hat M$ and to compute a finite 
generating set for $\Aut_{P_l(\O)}(M)$.

\mbox{} \\
{\bf Algorithm III.1:} \\
Decide if $M$ and $\hat{M}$ are isomorphic as $P_l(\O)$-modules and if
so, then determine an isomorphism.
\begin{items}
\item[$\bullet$]
Determine the unique maximal $P_l(\O_K)$-submodules $L$ and $\hat L$ of $M$
and $\hat M$ respectively (see Remark~\ref{findoksub}).
\item[$\bullet$]
If $[M:L] \neq [\hat{M} : \hat{L}]$ then return false.
\item[$\bullet$]
Choose $c \in \N$ such that $cM \leq L$. 
\item[$\bullet$]
Use Algorithm~II.1 to either construct a $P_l(\O_K)$-module isomorphism 
$\gamma \colon L \to \hat{L}$, or to conclude that no isomorphism exists 
and return false.
\item[$\bullet$]
Compute a finite generating set for $\Aut_{P_l(\O_K)}(L)$ using 
Algorithm~II.2.
\item[$\bullet$]
Compute the orbit $O$ of $cM$ under $\Aut_{P_l(\O_K)}(L)$; since $O$ is 
finite it can be listed explicitly.  Determine a transversal $T$ for $O$.
\item[$\bullet$]
Determine $\{ \gamma(w) \mid w \in O \}$ and decide if $c\hat M$ is 
contained in this set.  If not, then return false.
\item[$\bullet$]
Take $\tau \in T$ with $c \hat{M} = \gamma(\tau(cM))$. 
\item[$\bullet$]
Extend $\gamma \tau : cM \ra c\hat{M}$ to an isomorphism $M \ra \hat{M}$ 
using division by $c$ and return this isomorphism.
\end{items}

\mbox{} \\
{\bf Algorithm III.2:} \\
Determine generators for $\Aut_{P_l(\O)}(M)$.
\begin{items}
\item[$\bullet$]
Determine the unique maximal $P_l(\O_K)$-submodule $L$ of $M$.
\item[$\bullet$]
Choose $c \in \N$ such that $cM \leq L$. 
\item[$\bullet$]
Compute generators $\MySigma$ for the stabiliser of $cM$ in 
$\Aut_{P_l(\O_K)}(L)$ using Algorithm II.2.
\item[$\bullet$]
Extend each element in $\MySigma$ from $c M \ra cM$ to $M \ra M$ 
using division by $c$.
\item[$\bullet$]
Return the resulting set.
\end{items}

\section{The integral conjugacy and centraliser problems}

In Section~\ref{transl} we translated 
the conjugacy and centraliser problem to module theory
over truncated polynomial rings. 
We now formulate the two algorithms which 
exploit the module algorithms of Section \ref{module-algorithms}
to solve these problems.  

\mbox{} \\ 
{\bf Main Algorithm 1: Conjugacy algorithm.} \\ 
Decide if $T, \hat{T} \in \GL(n, \Q)$ are conjugate in $\GL(n,\Z)$ and if
so, then determine a conjugating element. 
\begin{items}
\item[$\bullet$]
Decide if $T$ and $\hat{T}$ are conjugate in $\GL(n,\Q)$; 
if not, then return false.
\item[$\bullet$]
Choose $k$ so that $kT$ and $k \hat{T}$ have integral Jordan-Chevalley
decompositions; 
replace $T$ and $\hat{T}$ by $kT$ and~$k \hat{T}$.
\item[$\bullet$]
Let $T = S + U$ and $\hat{T} = \hat{S} + \hat{U}$ be the Jordan-Chevalley
decompositions of $T$ and $\hat{T}$.
\item[$\bullet$]
Let $l \in \N$ be minimal with $U^l = \hat{U}^l = 0$.
\item[$\bullet$]
Let $P(x) \in \Z[x]$ be the minimal polynomial of $S$ and $\hat{S}$.
\item[$\bullet$]
Factorise $P(x) = P_1(x) \cdots P_r(x)$ with $P_i(x) \in \Z[x]$ monic
and irreducible over $\Q[x]$.
\item[$\bullet$]
Compute ${p}_i = P(x) / P_i(x)$ for $1 \leq i \leq r$.
\item[$\bullet$]
For $i \in \{1, \dots, r\}$ do:
\begin{items}
\item[-]
Compute $M_i = M ({p}_i(S))$ and $\hat{M}_i = \hat{M} ({p}_i(\hat{S}))$.
\item[-]
Let $\O_i = \Z[x]/(P_i)$.  Use Algorithm III.1 to either 
construct a $P_l(\O_i)$-module isomorphism $\gamma_i : M_i \ra \hat{M}_i$, 
or to conclude that no isomorphism exists and return false. 
\item[-]
Compute generators for $A_i = \Aut_{P_l(\O_i)}(M_i)$ using Algorithm III.2.
\end{items}
\item[$\bullet$]
Compute $D = M_1 + \dotsb + M_r$ and $\hat D = \hat M_1 + \dotsb + \hat M_r$.
\item[$\bullet$]
Choose $c \in \N$ such that $cM \leq D$.
\item[$\bullet$]
Combine $\gamma_1, \dots, \gamma_r$ to obtain $\gamma : D \ra \hat{D}$.
\item[$\bullet$]
Construct $\Aut_{P_l(R)}(D) = A_1 \times \dots \times A_r$.
\item[$\bullet$]
Compute the orbit $O$ of $cM$ under the action of $\Aut_{P_l(R)}(D)$;
since $O$ is finite it can be listed explicitly. 
Determine a transversal $T$ for $O$.
\item[$\bullet$]
Determine $\{ \gamma(w) \mid w \in O \}$ and decide if $c \hat{M}$ is
contained in this set. 
\item[$\bullet$]
If not, then return false.
\item[$\bullet$]
Take $\tau \in T$ with $c \hat{M} = \gamma(\tau(cM))$. 
\item[$\bullet$]
Extend $\gamma \tau : cM \ra c \hat{M}$ to an isomorphism $M \ra \hat{M}$ 
using division by $c$ and return this isomorphism.
\end{items}

\mbox{} \\ 
{\bf Main Algorithm 2: Centraliser algorithm.}  \\
Determine generators for $C_\Z(T)$ for $T \in \GL(n,\Q)$.
\begin{items}
\item[$\bullet$]
Choose $k$ so that $kT$ has an integral Jordan-Chevalley decomposition; 
replace $T$ by~$kT$.
\item[$\bullet$]
Let $T = S + U$ be the Jordan-Chevalley decomposition of $T$.
\item[$\bullet$]
Let $l \in \N$ be minimal with $U^l = 0$.
\item[$\bullet$]
Let $P(x) \in \Z[x]$ be the minimal polynomial of $S$.
\item[$\bullet$]
Factorise $P(x) = P_1(x) \cdots P_r(x)$ with $P_i(x) \in \Z[x]$ monic
and irreducible over $\Q[x]$.
\item[$\bullet$]
Compute ${p}_i = P(x) / P_i(x)$ for $1 \leq i \leq r$.
\item[$\bullet$]
For $i \in \{1, \dots, r\}$ do: 
\begin{items}
\item[-]
Compute $M_i = M ({p}_i(S))$.
\item[-]
Compute generators for $A_i = \Aut_{P_l(\O_i)}(M_i)$ using Algorithm III.2. 
\end{items}
\item[$\bullet$]
Compute $D = M_1 + \dotsb + M_r$.
\item[$\bullet$]
Choose $c \in \N$ such that $cM \leq D$.
\item[$\bullet$]
Construct $\Aut_{P_l(R)}(D) = A_1 \times \dots \times A_r$.
\item[$\bullet$]
Compute generators $\MySigma$ for the stabiliser of $cM$ in $\Aut_{P_l(R)}(D)$.
\item[$\bullet$]
Extend each element in $\MySigma$ from $c M \ra cM$ to $M \ra M$ using division 
by $c$.
\item[$\bullet$]
Return $\MySigma$.
\end{items}

\section{Implementation and performance}

We have implemented our algorithms in {\sc Magma}. We believe that 
this is the first
implementation which solves the integral conjugacy problem for 
arbitrary elements of $\GL(n,\Q)$.
Note that if $N, N' \in M_n(\Q)$ are nilpotent matrices, then $N$ are $N'$ 
are conjugate in $M_n(\Q)$ 
if and only if the invertible matrices $I + N, I + N' \in 
\GL(n, \Q)$ are conjugate. Thus we can also solve 
the integral conjugacy problem for nilpotent matrices.

To establish that a 
module defined over the maximal order of a number field is not free is a hard
problem, both in theory and practice.
Under the assumption of the generalised Riemann hypothesis (GRH),
there exist fast practical algorithms to solve this problem.
Our implementation allows us optionally to assume GRH. 
Note that a positive answer is always verifiable (independent of GRH), 
since the algorithm returns a conjugating element.

\subsection{Some examples}

The practical performance of our algorithms depends heavily on the 
structure of the input matrices. The limitations are not related to the 
dimension: our implementations sometimes work readily 
for ``random" elements of $\GL(100, \Q)$ but fail for elements of $\GL(10,\Q)$. 
The following computations were carried out on an Intel E5-2643 with 3.40GHz
and {\sc Magma} version V2.23-3 assuming GRH.

\begin{example}
Consider the conjugate matrices 
\[
\left(\begin{smallmatrix}
 0 &  1 &  0 & 0\cr
-4 &  0 &  0 & 0\cr
 0 &  0 &  0 & 1\cr
 0 &  0 & -4 & 0
\end{smallmatrix}\right), \quad
\left(\begin{smallmatrix}
 0 & 1 &  4 &  0\cr
-4 & 0 &  0 & -4\cr
 0 & 0 &  0 &  1\cr
 0 & 0 & -4 &  0
\end{smallmatrix}
\right).
\]
Our implementation takes $1$ second to find a conjugating matrix.
\end{example}

\begin{example}
Consider the conjugate matrices 
\[
\left(
  \begin{smallmatrix}
-14 & -4 &  0 &  0 & -1 &  1 &  0 &  0 &  0 &  0\cr
 -7 & -2 &  0 &  0 &  0 &  0 &  1 &  0 &  0 &  0\cr
 -3 & -1 &  0 &  0 &  0 &  0 &  0 &  1 &  0 &  0\cr
  0 &  0 & -1 &  0 &  0 &  0 &  0 &  0 &  1 &  0\cr
  0 &  0 &  0 & -1 &  0 &  0 &  0 &  0 &  0 &  1\cr
  0 &  0 &  0 &  0 &  0 &-14 & -4 &  0 &  0 & -1\cr
  0 &  0 &  0 &  0 &  0 & -7 & -2 &  0 &  0 &  0\cr
  0 &  0 &  0 &  0 &  0 & -3 & -1 &  0 &  0 &  0\cr
  0 &  0 &  0 &  0 &  0 &  0 &  0 & -1 &  0 &  0\cr
  0 &  0 &  0 &  0 &  0 &  0 &  0 &  0 & -1 &  0
\end{smallmatrix}
\right),
\left(
\begin{smallmatrix}
-9 & 9 & 0& -1&  0 & 0 & 0 & 0&  0& -7\cr
 0 & 0 & 0& -1&  0 & 0 & 0 & 0&  0&  0\cr
-4 & 4 & 0&  0&  0 & 0 & 0 & 0&  0& -3\cr
 0 & 0 &-1&  0&  0 & 0 & 0 & 0&  0&  0\cr
 0 & 0 & 0&  0& -7 &-9 & 1 & 0&  0&  0\cr
-1 & 1 & 0&  0& -7 &-9 & 0 & 0&  0&  0\cr
 9 &-7 & 0&  0&  0 & 0 & 0 & 0&  1&  6\cr
 0 & 0 & 1&  0&  3 & 4 & 0 & 0&  0&  0\cr
-9 & 8 & 0& -1&  0 & 0 & 0 & 1&  0& -7\cr
-9 & 8 & 0&  0&  0 & 0 & 0 & 0&  0& -7
\end{smallmatrix}
\right).
  \]
The minimal polynomial is $(x^5 + 16x^4 - 3x + 1)^2$. 
Our implementation takes $8$ seconds to find a conjugating matrix.
\end{example}
\begin{example}
Consider the conjugate matrices: 
\[
  \left(
 \begin{smallmatrix}
 0&  1&  0&  0 & 0&  0&  0&  0&  0&  0&  0&  0&  0&  0\cr
 0&  0&  1&  0 & 0&  0&  0&  0&  0&  0&  0&  0&  0&  0\cr
 0&  0&  0&  1 & 0&  0&  0&  0&  0&  0&  0&  0&  0&  0\cr
 0&  0&  0&  0 & 1&  0&  0&  0&  0&  0&  0&  0&  0&  0\cr
 0&  0&  0&  0 & 0&  1&  0&  0&  0&  0&  0&  0&  0&  0\cr
 0&  0&  0&  0 & 0&  0&  1&  0&  0&  0&  0&  0&  0&  0\cr
 0&  0&  0&  0 & 0&  0&  0&  1&  0&  0&  0&  0&  0&  0\cr
 0&  0&  0&  0 & 0&  0&  0&  0&  1&  0&  0&  0&  0&  0\cr
 0&  0&  0&  0 & 0&  0&  0&  0&  0&  1&  0&  0&  0&  0\cr
 0&  0&  0&  0 & 0&  0&  0&  0&  0&  0&  1&  0&  0&  0\cr
 0&  0&  0&  0 & 0&  0&  0&  0&  0&  0&  0&  1&  0&  0\cr
 0&  0&  0&  0 & 0&  0&  0&  0&  0&  0&  0&  0&  1&  0\cr
 0&  0&  0&  0 & 0&  0&  0&  0&  0&  0&  0&  0&  0&  1\cr
-2& -2& -2& -2 &-1& -1& -1& -1&  0&  0& -2&  0&  0&  0
\end{smallmatrix}
\right),
\left(
\begin{smallmatrix}
-21 & -11 &  12 &   9 &  15 &   8 & -22 & -23 &   2 &   6 &   8 &  -2 &   2 &  -8\cr
 32 &  17 & -22 & -10 & -18 & -14 &  28 &  35 &  -5 &  -6 & -13 &   4 &  -1 &  13\cr
  1 &   4 & -17 &   6 &  11 & -12 &  -6 &  18 & -16 &   8 &  -7 &   7 &   4 &   9\cr
 21 &  13 & -24 & -11 & -19 & -21 &  33 &  48 & -15 &  -4 & -15 &   7 &  -1 &  18\cr
 14 &   8 & -16 &  -2 &  -2 & -11 &  10 &  23 &  -9 &   0 &  -9 &   5 &   0 &   9\cr
 -1 &   8 & -14 &  -3 &  13 &  -9 &  -2 &  21 & -18 &   6 &  -7 &   8 &   0 &   9\cr
 29 &  14 & -13 & -15 & -27 & -10 &  36 &  31 &   1 & -11 & -10 &   1 &  -4 &  10\cr
-26 &  -8 &  -8 &  10 &  21 & -12 & -16 &  14 & -25 &  14 &  -3 &   9 &   5 &   8\cr
 -8 &  -6 &  -1 &  -9 & -31 & -14 &  34 &  31 &  -8 &  -7 &  -4 &   1 &  -2 &  10\cr
-19 &  -8 &  11 & -15 & -39 &  -8 &  38 &  24 &  -4 &  -9 &   2 &  -2 &  -3 &   7\cr
-35 & -25 &  30 &   3 & -11 &  11 &  -2 & -27 &  14 &  -4 &  14 &  -9 &  -2 & -13\cr
-19 & -21 &  32 &  -6 & -34 &  11 &  21 & -21 &  23 & -14 &  13 & -13 &  -5 & -12\cr
 10 &   5 &   3 &  11 &  31 &  16 & -36 & -36 &  11 &   7 &   5 &  -2 &   3 & -12\cr
 12 &   7 &  -7 & -13 & -27 &  -7 &  30 &  24 &   2 & -11 &  -5 &  -1 &  -5 &   6
\end{smallmatrix}
\right).
\]
The minimal polynomial is $(x^4 + 2)(x^{10} + x^3 + x^2 + x + 1)$.
Our implementation takes $18$ seconds to find a conjugating matrix.
\end{example}

\begin{example} 
Consider 
\[T = \left(
\begin{matrix}
-5 &  8 & -5 \cr
 4 & -7 &  5 \cr
 1 & -2 &  2
\end{matrix}
\right).
\]
Our implementation shows in $0.3$ seconds that
\[ C_\Z(T) = \langle 
\left(\begin{matrix}
 860 & 1206 & -975\cr
 603 & 1001 & -795\cr
 195 &  318 & -253
\end{matrix}\right),
\left(\begin{matrix}
 4 &  6 & -5\cr
 3 &  5 & -5\cr
 1 &  2 & -3
\end{matrix}\right),
\left(\begin{matrix}
-1 & 0 & 0\cr
0 & -1 & 0\cr
0 & 0 & -1
\end{matrix}\right)\rangle.
\]
\end{example}

\subsection{Comparison with other implementations}

Kirschmer implemented in {\sc Magma} the algorithms of \cite{OPSc1998}
which solve the integral conjugacy and centraliser problems for matrices
of finite order. By Maschke's theorem, such matrices are semisimple. 
While these algorithms work well for small matrices with small entries, 
they are very sensitive to entry size. For example,
his implementation  took 380 seconds and ours 2 seconds
to decide that the following matrices are conjugate: 
\[\left(
\begin{smallmatrix}
 2 &   1 &  -1  & -1 &   2\cr
 6 &  -2 &   4  &  3 &   2\cr
10 &  -5 &   4  &  5 & -13\cr
22 &   6 &  -2  & -4 &  23\cr
-1 &   0 &   0  &  0 &   0
\end{smallmatrix}
\right), \quad 
\left(\begin{smallmatrix}
  85 &   -89 &  -167 &    22 &    -2\cr
9480 &  9317 & 17095 &  -307 &  -214\cr
5233 & -5146 & -9444 &   180 &   116\cr
1045 & -1028 & -1887 &    38 &    23\cr
  52 &   -47 &   -84 &   -10 &     4
\end{smallmatrix}
\right).
\]


Husert implemented in {\sc Magma} his algorithm \cite{Hus16} which solves
the integral conjugacy problem for rational matrices which have an 
irreducible characteristic polynomial or are nilpotent.
Also Marseglia implemented in {\sc Magma}
his algorithms \cite{Mars2018} for the case of 
squarefree characteristic polynomial.
Both implementations usually outperform our general-purpose method. 
The developers have kindly allowed us to incorporate their code 
into our implementation.

\subsection{Practical limitations}

We identify various limitations to our algorithms and then discuss them 
in more detail.
\begin{items}
\item[(1)] 
Constructing the maximal order and ideal class group of certain number 
fields is hard. These are used to decide if  an integral 
$P_l(\O_K)$-module is standard. 
\item[(2)]
The finite orbits arising in Algorithms 
II.1-2, III.1-2, and in 
Main Algorithms 1-2 are too long to list explicitly.
\item[(3)]
The number of standard submodules constructed using Algorithm~I.3 is 
too large to list them explicitly.
\end{items}

\subsubsection{Computing maximal orders and ideal class groups} 

We use classical methods from algorithmic number theory for such 
computations, see \cite[Chapter 6]{Co1993}.
These work well if the discriminant of each irreducible
factor of the minimal polynomial of a matrix is not too large
and can be factorised.
But the minimal polynomial of the following matrix
has discriminant of size $\sim 10^{108}$, 
and computing the maximal order is not feasible.
\[
\left(
\begin{smallmatrix}
 6 & -8 & -4 & -2 &  3 &  8 & -2 &  3 & -1 &  7\cr
 2 &  2 & -6 &  6 &  6 & -1 &  3 &  7 &  1 &  0\cr
 8 & -2 & -1 &  1 & 10 & -3 & -3 & -2 & -2 &  3\cr
 1 & 10 &  1 &-10 &  3 &  5 & -5 &-10 & -7 & -6\cr
 1 &  0 &  3 & -2 &  0 &  6 &  4 &  1 &  1 & -4\cr
 2 & -3 &  9 &  4 & -2 & -8 & -8 &  4 &  4 & -4\cr
 5 & -1 & -4 & -7 & -8 &  8 &  1 &  3 & -6 & 10\cr
-6 & -2 & -7 &  5 & 10 & -8 &  6 &  3 & -8 & -6\cr
-7 & 10 & -5 &  4 &  2 &  3 & -7 &  7 & -8 & -3\cr
-1 &  1 & -3 &  0 &  2 & -9 & -6 & -1 & -6 & -6
\end{smallmatrix}
\right)
\]

\subsubsection{Large orbits} 

Finite orbits and their associated stabilisers are 
constructed using general group theoretic algorithms, 
see \cite[Section 4.1]{HEO05}.
Their success depends heavily on the length of the orbit. 
Consider the following conjugate matrices
\[
{ 
\left(
\begin{smallmatrix}
-3 & -1 & 3 &  0 &  0 & 0 &  0 &  0 & 0\cr
 1 &  0 & 0 &  0 &  0 & 0 &  0 &  0 & 0\cr
-5 &  0 & 1 &  0 &  0 & 0 &  0 &  0 & 0\cr
 0 &  0 & 0 & -3 & -1 & 3 &  0 &  0 & 0\cr
 0 &  0 & 0 &  1 &  0 & 0 &  0 &  0 & 0\cr
 0 &  0 & 0 & -5 &  0 & 1 &  0 &  0 & 0\cr
 0 &  0 & 0 &  0 &  0 & 0 & -3 & -1 & 3\cr
 0 &  0 & 0 &  0 &  0 & 0 &  1 &  0 & 0\cr
 0 &  0 & 0 &  0 &  0 & 0 & -5 &  0 & 1
\end{smallmatrix}
\right),
\left(
\begin{smallmatrix}
 13 & -15 &  16 &   24 &  -16 &  -7 &  -35 &   15 &   0\cr
 -3 &  44 & -40 &  -71 &   62 &  28 &  157 &  -76 &  16\cr
 18 & -15 &  -3 &   -7 &  -31 &   6 & -226 &  129 & -52\cr
-69 &  72 & -55 &  -78 &   86 &  18 &  355 & -186 &  48\cr
-75 &  98 & -82 & -124 &  117 &  35 &  406 & -206 &  46\cr
-45 &  19 & -21 &  -22 &   10 &   1 &   49 &  -25 &  -3\cr
 24 & -66 &  53 &   89 &  -89 & -31 & -289 &  147 & -37\cr
 30 & -78 &  61 &  102 & -104 & -35 & -348 &  178 & -45\cr
 24 &  11 &  -8 &  -23 &   26 &  14 &   58 &  -29 &  11
\end{smallmatrix}
\right),
}
\]
each having characteristic polynomial $(x^3 + x^2 + 13x - 1)^3$. The orbit
of $cM$ in Algorithm III.1 contains at least $382000$ elements. 

\subsubsection{Large numbers of standard submodules}

Consider the conjugate matrices
\[
\left(\begin{smallmatrix}
  13 &    67 & 6 & 0 & 0 & -1\cr
   0 &     1 & 3 & 0 & 0 &  0\cr
   0 &     0 & 1 & 0 & 0 &  0\cr
-270 & -1350 & 0 & 1 & 2 & 20\cr
-135 &  -675 & 0 & 0 & 1 & 10\cr
 -27 &  -135 & 0 & 0 & 0 &  2
\end{smallmatrix}
\right),\quad
\left(\begin{smallmatrix}
  13 &    79 & 0 & 0 &   1 &  -76\cr
   0 &     1 & 0 & 0 &   0 &    3\cr
-270 & -1620 & 1 & 2 & -20 & 1620\cr
-135 &  -810 & 0 & 1 & -10 &  810\cr
  27 &   162 & 0 & 0 &   2 & -162\cr
   0 &     0 & 0 & 0 &   0 &    1
\end{smallmatrix}
\right).
\]
Algorithm~I.3 finds 7144200 standard submodules. 

\section{Variations and open problems}

We may ask for solutions to the conjugacy 
problem in either of $\SL(n,\Z)$ or $\PGL(n, \Z)$. 

\begin{lemma}
If we can solve the integral conjugacy and centraliser problems in 
$\GL(n,\Z)$, then we can solve the conjugacy problem in $\SL(n,\Z)$ 
or in $\PGL(n, \Z)$. 
\end{lemma}

\begin{proof}
Let $T, \hat{T} \in \GL(n,\Q)$ and consider $\SL(n,\Z)$.  
If $T$ and $\hat{T}$ are not conjugate in $\GL(n,\Z)$, then
they are not conjugate in $\SL(n,\Z)$. Otherwise, we obtain 
$X \in \GL(n,\Z)$ with $X T X^{-1} = \hat{T}$. 
If $\det(X) = 1$, then we know a conjugating element in $\SL(n, \Z)$. 
Now suppose $\det(X)=-1$. If $n$ is odd, then $-X \in \SL(n, \Z)$ 
conjugates $T$ to $\hat{T}$. If $n$ is even, then we check if a generator, 
say $g$, of $C_\Z(T)$ has determinant $-1$. If so,
$g X \in \SL(n, \Z)$ conjugates $T$ to $\hat{T}$; otherwise they are not conjugate in $\SL(n, \Z)$. 

Now consider $\PGL(n,\Z)$. Let $Z = \{ I, -I\}$ denote the center of 
$\GL(n,\Z)$. Now $TZ$ and $\hat{T} Z$ are conjugate in $\PGL(n,\Z)$
if and only if there exists $X \in \GL(n,\Z)$ with $X T X^{-1} 
\equiv \hat{T} \bmod Z$. Equivalently either $X T X^{-1} = \hat{T}$ 
or $X T X^{-1} = -\hat{T}$; so we can decide using the solution to  
the integral conjugacy problem in $\GL(n,\Z)$.
\end{proof}

We conclude by identifying related problems of interest. 

\begin{problem}\label{prob-one}
Given $T \in \GL(n,\Q)$, determine a finite presentation for $C_\Z(T)$.
\end{problem}
One possible approach to this challenging problem is to extend our algorithm 
so that finite presentations are constructed in each step.

\begin{problem}\label{prob-two}
Define a canonical form for the conjugacy classes in $\GL(n,\Z)$
and provide a practical algorithm to compute it.
\end{problem}
By contrast, each conjugacy class in $\GL(n,\Q)$ is represented by a 
unique rational canonical form and this can be determined effectively
by a variation of the Gauss algorithm. 

\begin{problem}\label{prob-three}
Solve the conjugacy and the centraliser problem for $\GL(n, K[x])$,
where $K[x]$ is a polynomial ring over a field $K$.
\end{problem}
Since the ring structures of $\Z$ and $K[x]$ are similar, one 
potential approach 
is to translate our algorithms to $\GL(n, K[x])$. 
Grunewald and Iyudu \cite{GrIy05} show how to solve the conjugacy 
problem in $\GL(2, K[x])$, where $K$ is finite; and claim 
that the methods of \cite{Gru80} extend to $\GL(n,K[x])$ where 
$K$ is finite and of characteristic coprime to $n$.

\begin{problem}\label{prob-four}
$T \in \GL(n, \Z)$ is {\em conjugacy distinguished} if 
every $\hat{T} \in \GL(n, \Z)$ is either conjugate to $T$ or there exists a
homomorphism $\varphi$ onto a finite quotient of $\GL(n,\Z)$ in which 
$\varphi(T)$ is not conjugate to $\varphi(\hat{T})$. 
Decide whether or not $T \in \GL(n,\Z)$ is conjugacy distinguished.
\end{problem}
Stebe \cite{Ste72} shows that there 
exist elements in $\GL(n, \Z)$ for all $n \geq 3$ that are not 
conjugacy distinguished. 

\begin{problem}\label{prob-five}
Devise a practical solution to the conjugacy problem in $\GL(n,\O_K)$,
where $\O_K$ is the maximal order of an algebraic number field $K$. 
\end{problem}
Both Grunewald \cite{Gru80} and Sarkisyan \cite{Sar79} state that
it is possible to extend their methods to this case. 
It remains both to verify these claims and to realise them practically.

\begin{problem}\label{prob-six}
Given semisimple $T \in \GL(n,\Q)$, devise an algorithm to
list a complete and irredundant set of representatives for the 
$\GL(n,\Z)$-classes in the $\GL(n,\Q)$-class of $T$.
\end{problem}
The $\GL(n,\Q)$-class of an arbitrary $T \in \GL(n,\Q)$ is the disjoint
union of $\GL(n,\Z)$-classes. The Jordan-Zassenhaus theorem \cite{Zassenhaus37}
shows that this union is finite if and only if $T$ is semisimple. 
If $T$ has an irreducible or squarefree characteristic polynomial,
then 
Problem \ref{prob-six} is solved 
in \cite{LMa1933} and \cite{Mars2018} respectively.

\bibliographystyle{plain}

\def\cprime{$'$}

\end{document}